\documentclass{article}
\usepackage[LGR,T1]{fontenc}
\usepackage[latin9]{inputenc}
\usepackage{color}
\usepackage{array}
\usepackage{url}
\usepackage{amsmath}
\usepackage{amsthm}
\usepackage{amssymb}
\usepackage{graphicx}
\usepackage[unicode=true,
 bookmarks=true,bookmarksnumbered=true,bookmarksopen=true,bookmarksopenlevel=3,
 breaklinks=false,pdfborder={0 0 0},pdfborderstyle={},backref=false,colorlinks=true]
 {hyperref}
\hypersetup{pdftitle={Renyi},
 pdfauthor={Lei Yu},
 pdfborderstyle={},pdfborderstyle={},pdfborderstyle={},pdfborderstyle={},pdfborderstyle={},pdfborderstyle={},pdfborderstyle={},pdfborderstyle={},pdfborderstyle={},pdfborderstyle={},pdfborderstyle={},pdfborderstyle={},pdfborderstyle={},pdfborderstyle={},pdfborderstyle={},pdfborderstyle={},pdfpagelayout=OneColumn,pdfnewwindow=true,pdfstartview=XYZ,plainpages=false,linkcolor=blue,urlcolor=blue,citecolor=red,anchorcolor=blue,linkcolor=blue,urlcolor=blue,citecolor=red,anchorcolor=blue}

\makeatletter

\providecommand{\tabularnewline}{\\}
\theoremstyle{plain}
\newtheorem{thm}{\protect\theoremname}[section]
\theoremstyle{definition}
\newtheorem{defn}[]{\protect\definitionname}[section]
\theoremstyle{plain}
\newtheorem{lem}[]{\protect\lemmaname}[section]
\theoremstyle{plain}
\newtheorem{prop}[]{\protect\propositionname}[section]
\theoremstyle{plain}
\newtheorem{cor}[]{\protect\corollaryname}[section]
\theoremstyle{plain}
\newtheorem{rem}[]{\protect\remarkname}[section]
\theoremstyle{plain}
\newtheorem{example}[]{\protect\examplename}[section]
\theoremstyle{plain}
\newtheorem{claim}[]{\protect\claimname}[section]

\usepackage{color}
\usepackage{cite}
\usepackage{fullpage} 
\usepackage{url}

\usepackage{tikz}

\DeclareMathOperator*{\esssup}{\mathrm{ess\,sup}}

\DeclareMathAlphabet{\mathbsf}{OT1}{cmss}{bx}{n}
\DeclareMathAlphabet{\mathssf}{OT1}{cmss}{m}{sl}

\DeclareSymbolFont{bsfletters}{OT1}{cmss}{bx}{n}
\DeclareSymbolFont{ssfletters}{OT1}{cmss}{m}{n}
\DeclareMathSymbol{\bsfGamma}{0}{bsfletters}{'000}
\DeclareMathSymbol{\ssfGamma}{0}{ssfletters}{'000}
\DeclareMathSymbol{\bsfDelta}{0}{bsfletters}{'001}
\DeclareMathSymbol{\ssfDelta}{0}{ssfletters}{'001}
\DeclareMathSymbol{\bsfTheta}{0}{bsfletters}{'002}
\DeclareMathSymbol{\ssfTheta}{0}{ssfletters}{'002}
\DeclareMathSymbol{\bsfLambda}{0}{bsfletters}{'003}
\DeclareMathSymbol{\ssfLambda}{0}{ssfletters}{'003}
\DeclareMathSymbol{\bsfXi}{0}{bsfletters}{'004}
\DeclareMathSymbol{\ssfXi}{0}{ssfletters}{'004}
\DeclareMathSymbol{\bsfPi}{0}{bsfletters}{'005}
\DeclareMathSymbol{\ssfPi}{0}{ssfletters}{'005}
\DeclareMathSymbol{\bsfSigma}{0}{bsfletters}{'006}
\DeclareMathSymbol{\ssfSigma}{0}{ssfletters}{'006}
\DeclareMathSymbol{\bsfUpsilon}{0}{bsfletters}{'007}
\DeclareMathSymbol{\ssfUpsilon}{0}{ssfletters}{'007}
\DeclareMathSymbol{\bsfPhi}{0}{bsfletters}{'010}
\DeclareMathSymbol{\ssfPhi}{0}{ssfletters}{'010}
\DeclareMathSymbol{\bsfPsi}{0}{bsfletters}{'011}
\DeclareMathSymbol{\ssfPsi}{0}{ssfletters}{'011}
\DeclareMathSymbol{\bsfOmega}{0}{bsfletters}{'012}
\DeclareMathSymbol{\ssfOmega}{0}{ssfletters}{'012}

\def\dotle{\mathrel{\dot{\le}}}

\def\dotge{\mathrel{\dot{\ge}}}

\def\dotpreceq{\mathrel{\dot{\preceq}}}

\def\dotsucceq{\mathrel{\dot{\succeq}}}
\def\dotasymp{\mathrel{\dot{\asymp}}}

\allowdisplaybreaks
\providecommand{\keywords}[1]{\textbf{{Index terms:}} #1}
\providecommand{\subjclass}[1]{\textbf{{2000 Mathematics subject classification:}} #1}
\providecommand{\claimname}{Claim}
\providecommand{\corollaryname}{Corollary}
\providecommand{\definitionname}{Definition}
\providecommand{\examplename}{Example}
\providecommand{\lemmaname}{Lemma}
\providecommand{\propositionname}{Proposition}
\providecommand{\remarkname}{Remark}
\providecommand{\theoremname}{Theorem}

\@ifundefined{showcaptionsetup}{}{%
 \PassOptionsToPackage{caption=false}{subfig}}
\usepackage{subfig}
\makeatother

\providecommand{\claimname}{Claim}
\providecommand{\corollaryname}{Corollary}
\providecommand{\definitionname}{Definition}
\providecommand{\examplename}{Example}
\providecommand{\lemmaname}{Lemma}
\providecommand{\propositionname}{Proposition}
\providecommand{\remarkname}{Remark}
\providecommand{\theoremname}{Theorem}

\begin{document}

\title{Beyond the Central Limit Theorem: Universal and Non-universal Simulations
of Random Variables by General Mappings}

\author{Lei Yu\\
 Department of Electrical and Computer Engineering \\
 National University of Singapore, Singapore\\
 Email: leiyu@nus.edu.sg}
\maketitle
%
\maketitle
\begin{abstract}
Motivated by the  Central Limit Theorem, in this paper, we study  both universal and non-universal simulations of random variables with an arbitrary target distribution $Q_{Y}$ by general mappings, not limited to linear ones (as in the Central Limit Theorem). We derive the fastest convergence rate of the approximation errors for such  problems. Interestingly, we show that for discontinuous or absolutely continuous $P_{X}$, the approximation error for the universal simulation is almost as small as that for the non-universal one; and moreover, for both universal and non-universal simulations, the approximation errors by general mappings are strictly smaller than those by linear mappings. Furthermore, we also generalize these results to simulation from Markov processes, and simulation of random elements (or general random variables). 
\end{abstract}





\section{\label{sec:Introduction}Introduction}

The Central Limit Theorem (CLT) states that for a sequence of i.i.d.
real-valued random variables $X^{n}\sim P_{X}^{n}$, the normalized
sum $\frac{1}{\mathrm{Var}(X)\sqrt{n}}\sum_{i=1}^{n}\left(X_{i}-\mathbb{E}\left[X\right]\right)$
converges in distribution to a standard Gaussian random variable as
$n$ goes to infinity. This implies that an $n$-dimensional i.i.d.
random vector $X^{n}$ can be used to simulate a standard Gaussian
random variable $Y$ by the normalized sum so that the approximation
error asymptotically vanishes under the Kolmogorov\textendash Smirnov
distance. Moreover, from the Berry-Esseen theorem \cite[Sec. XVI.5]{feller2008introduction},
the approximation error vanishes in a rate of $\frac{1}{\sqrt{n}}$.
Note that here, the distribution $P_{X}$ of $X$ is arbitrary, and
given the mean and variance, the linear function is independent of
$P_{X}$. Hence such a linear function can be considered as a universal
linear function. The corresponding simulation problem can be considered
as being universal. In this paper, we consider general universal simulation
problems, in which general\footnote{We say a mapping is general if it is either linear or non-linear.}
simulation functions, not limited to linear ones, are allowed. We
are interested in the following question: What is the optimal convergence
rate for such universal simulation problems? To know how important
the knowledge of the distribution $P_{X}$ is in a simulation, we
are also interested in the optimal convergence rate for non-universal
simulation problems (in which $P_{X}$ is known). Is the optimal convergence
rate for universal simulation as fast as, or strictly slower than,
that for non-universal simulation?

The CLT is about universal simulation of a continuous random variable
(more specifically, a Gaussian random variable). In addition to simulation
of continuous random variables, there are a large number of works
that consider universal simulation of a sequence of discrete (or atomic)
random variables from another sequence of discrete random variables.
In 1951, von Neumann \cite{von1951various} described a procedure
for \emph{exactly} generating a sequence of independent and identically
distributed (i.i.d.) unbiased random coins from a sequence of i.i.d.
biased random coins with an \emph{unknown} distribution. To obtain
unbiased outputs, two pairs of bits $(0,1)$ and $(1,0)$ (which have
the same empirical distribution) are mapped to $0$ and $1$, respectively,
and $(0,0)$ and $(1,1)$ are discarded. Elias \cite{elias1972efficient}
and Blum \cite{blum1986independent} considered a more general situation
in which the process of the repeated coin tosses is subject to an
\emph{unknown} Markov process, instead of a traditional i.i.d. process,
and then studied the efficiency of such a procedure measured according
to the expected number of output coins per input coin. Knuth and Yao
\cite{knuth1976complexity}, Roche \cite{roche1991efficient}, Abrahams
\cite{abrahams1994generation}, and Han and Hoshi \cite{Han1997interval}
considered another general simulation problem in which an arbitrary
target distribution is generated by using a unbiased or biased $M$-coin
(i.e., an $M$-sided coin) but with a \emph{known} distribution. They
showed that the minimum expected number of coin tosses required to
generate the target distribution can be expressed in terms of the
ratio of the entropy of the target distribution to that of the seed
distribution. In all of the works above \cite{von1951various,elias1972efficient,blum1986independent,knuth1976complexity,roche1991efficient,abrahams1994generation,Han1997interval},
simulators are defined as functions that map a variable-length input
sequence to a fixed-length output sequence. Hence, to produce an output
symbol, arbitrarily long delay or waiting time may be required.

To reduce delay, a direction of generalizing the random number generation
problem is to require that an output must be generated for every $k$
bits input from a unbiased or biased coin, for any fixed $k$, but
at the same time, relax the requirement of exact generation to that
of approximate generation. That is, we may require only that the target
distribution should be generated \emph{approximately} within a nonzero
but arbitrarily small tolerance in terms of some suitable distance
measures such as the total variation distance or divergences. Such
a problem in the asymptotic context with \emph{known} seed and target
distributions has been formulated and studied by Han and Verdú \cite{Han};
its inverse problem has been investigated by Vembu and Verdú \cite{vembu1995generating};
and a general version of these problems \textemdash \textendash{}
generating an i.i.d. sequence from another i.i.d. sequence with arbitrary
\emph{known} seed and target distributions \textemdash \textendash{}
has been studied in \cite{Han03,kumagai2017second,yu2018simulation}.

All of the works above only considered simulating a sequence of \emph{discrete}
random variables from another sequence of \emph{discrete} random variables.
In contrast, in this paper we consider approximately\emph{ }generating
an \emph{arbitrary} random variable (or a random element) from a sequence
of random variables (or another random element) with arbitrary but
\emph{unknown} seed distribution.

Besides the CLT, this work is also motivated by the following questions.
1) Given a distribution $Q_{Y}$ (defined on $\left(\mathbb{R},\mathcal{B}_{\mathbb{R}}\right)$),
is there \emph{a measurable function} $f:\mathbb{R}\to\mathbb{R}$
such that $P_{f(X)}=Q_{Y}$ for all absolutely continuous distribution
$P_{X}$? Here $P_{f(X)}$ is the distribution of the image $f(X)$
induced by $P_{X}$ and the function $f$. 2) Given $Q_{Y}$, is there
\emph{a sequence of measurable functions} $f_{k}:\mathbb{R}\to\mathbb{R}$
such that $P_{f_{k}(X)}\to Q_{Y}$ as $k\to\infty$ (under the total
variation distance or other distance measures) for all absolutely
continuous distribution $P_{X}$? By some simple derivations, it is
easy to show that the answer to the first question is negative. So
it is intuitive to conjecture the answer to the second one is also
negative, since the second question reduces to the first question
if the limit of the sequence $\left\{ f_{k}\right\} $ is set to the
function $f$. However, the results in this paper show that this conjecture
is not right, since the limit of the optimal sequence $\left\{ f_{k}\right\} $
does not exist and hence these two questions are not equivalent. Interestingly, we
show that the answer to the second question is \emph{positive}. 

\subsection{Problem Formulation}

Before formulating our problem, we first introduce two statistical
distances. For an arbitrary measurable space $\left(\Omega,\mathcal{B}_{\Omega}\right)$,
we use $\mathcal{P}(\Omega,\mathcal{B}_{\Omega})$ to denote the set
of all the probability measures (a.k.a. distributions) defined on
$\left(\Omega,\mathcal{B}_{\Omega}\right)$. Given an arbitrary measurable
space $\left(\Omega,\mathcal{B}_{\Omega}\right)$, the \emph{total
variation (TV) distance} between two probability measures $P,Q\in\mathcal{P}(\Omega,\mathcal{B}_{\Omega})$
is defined as 
\begin{align*}
\left|P-Q\right|_{\mathrm{TV}} & =\sup_{A\in\mathcal{B}_{\Omega}}\left|P(A)-Q(A)\right|.
\end{align*}
The \emph{Kolmogorov\textendash Smirnov (KS) distance} between two
probability measures $P,Q\in\mathcal{P}(\mathbb{R},\mathcal{B}_{\mathbb{R}})$
is defined as 
\begin{align*}
\left|P-Q\right|_{\mathrm{KS}} & =\sup_{x\in\mathbb{R}}\left|F(x)-G(x)\right|,
\end{align*}
where $F$ and $G$ respectively denote the CDFs (cumulative distribution
functions) of $P$ and $Q$. For $P,Q\in\mathcal{P}(\mathbb{R},\mathcal{B}_{\mathbb{R}})$,
we have 
\[
0\le\left|P-Q\right|_{\mathrm{KS}}\leq\left|P-Q\right|_{\mathrm{TV}},
\]
since $\left|P-Q\right|_{\mathrm{KS}}=\sup_{A\in\mathcal{I}}\left|P(A)-Q(A)\right|$
with $\mathcal{I}:=\left\{ (-\infty,y]:\:y\in\mathbb{R}\right\} \subseteq\mathcal{B}_{\mathbb{R}}$.
Furthermore, both $\left|P-Q\right|_{\mathrm{KS}}$ and $\left|P-Q\right|_{\mathrm{TV}}$
are metrics, and hence $\left|P-Q\right|_{\mathrm{KS}}=0\;\Longleftrightarrow\;P=Q$
 and $\left|P-Q\right|_{\mathrm{TV}}=0\;\Longleftrightarrow\;P=Q$.

Based on these two distances, we next formulate our problem. In this
paper, we consider the following problem: When we use an $n$-dimensional
real-valued random vector $X^{n}$ with distribution $P_{X^{n}}$
to generate a real-valued random variable $Y$ by a function $y=f(x^{n})$
so that its distribution is approximately $Q_{Y}$, what is the fastest
convergence speed of the approximation error over all functions $f$
as $n$ tends to infinity? Here the approximation error is measured
by the TV distance or the KS distance. We term the Borel space $\left(\mathbb{R}^{n},\mathcal{B}_{\mathbb{R}^{n}}\right)$
of $X^{n}$ as the seed space, and the Borel space $\left(\mathbb{R},\mathcal{B}_{\mathbb{R}}\right)$
of $Q_{Y}$ as the target space.
\begin{defn}
Given the seed Borel space $\left(\mathbb{R}^{n},\mathcal{B}_{\mathbb{R}^{n}}\right)$
and the target Borel space $\left(\mathbb{R},\mathcal{B}_{\mathbb{R}}\right)$,
a simulator is a measurable function $f:\mathbb{R}^{n}\to\mathbb{R}$. 
\end{defn}
Given a random vector $X^{n}\sim P_{X^{n}}$ and a target distribution
$Q_{Y}$, we want to find an optimal simulator $Y=f(X^{n})$ that
minimizes the TV distance or the KS distance between the output distribution
$P_{Y}:=P_{X^{n}}\circ f^{-1}$ (the distribution of the output random
variable $Y$) and the target distribution $Q_{Y}$. For such a simulation
problem, we consider two different scenarios where $P_{X^{n}}$ is
respectively known and unknown a priori.

As illustrated in Fig. \ref{fig:simulation} (a), if the seed distribution
$P_{X^{n}}$ is \emph{unknown}, but the class $\mathcal{P}_{X^{n}}\subseteq\mathcal{P}(\mathbb{R}^{n},\mathcal{B}_{\mathbb{R}^{n}})$
that $P_{X^{n}}$ belongs to is known, we term such simulation problems
as \emph{(universal)} \emph{$(\mathcal{P}_{X^{n}},Q_{Y})$-simulation
problems}. Hence, the simulator $f:\mathbb{R}^{n}\to\mathbb{R}$ in
the universal simulation problem may depend on everything including
$Q_{X}$ and $\mathcal{P}_{X^{n}}$, but except for $P_{X^{n}}$.
That is, it is independent of $P_{X^{n}}$ given $\mathcal{P}_{X^{n}}$.
Next we give a mathematical formulation for the universal simulation
problem, which avoids ambiguous languages, like ``$P_{X^{n}}$ is
unknown''.

\begin{figure}
\centering\subfloat[The universal $(\mathcal{P}_{X^{n}},Q_{Y})$-simulation problem]{\centering \setlength{\unitlength}{0.05cm}\scalebox{1.2}{ \begin{picture}(140,50)
\put(5,15){\vector(1,0){35}} \put(40,5){\framebox(50,20){%
\parbox[c]{35in}{%
\centering Simulator \\
 $Y=f(X^{n})$%
}}} \put(90,15){\vector(1,0){45}} \put(65,40){\vector(0,-1){15}}
\put(5,20){%
\mbox{%
$X^{n}\sim P_{X^{n}}$%
}} \put(100,20){%
\mbox{%
$Y$ s.t. $P_{Y}\approx Q_{Y}$%
}} \put(48,45){%
\mbox{%
$\left(\mathcal{P}_{X^{n}},Q_{Y}\right)$%
}} \end{picture}} 

}

\subfloat[The non-universal $(P_{X^{n}},Q_{Y})$-simulation problem]{\centering \setlength{\unitlength}{0.05cm}\scalebox{1.2}{ \begin{picture}(140,50)
\put(5,15){\vector(1,0){35}} \put(40,5){\framebox(50,20){%
\parbox[c]{35in}{%
\centering Simulator \\
 $Y=f(X^{n})$%
}}} \put(90,15){\vector(1,0){45}} \put(65,40){\vector(0,-1){15}}
\put(5,20){%
\mbox{%
$X^{n}\sim P_{X^{n}}$%
}} \put(100,20){%
\mbox{%
$Y$ s.t. $P_{Y}\approx Q_{Y}$%
}} \put(48,45){%
\mbox{%
$\left(P_{X^{n}},Q_{Y}\right)$%
}} \end{picture}} 

}

\caption{Universal and non-universal simulation problems.}
\label{fig:simulation} 
\end{figure}
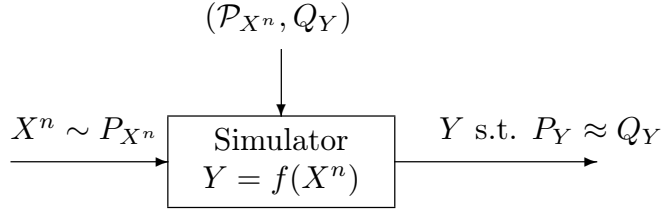
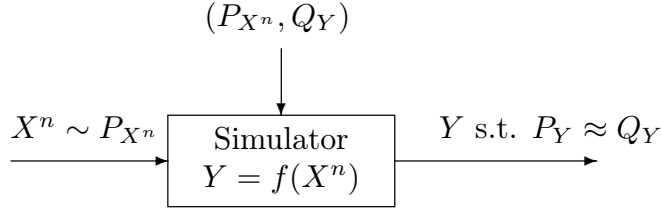
\begin{defn}
A function $g:\mathcal{P}_{X^{n}}\to\mathbb{R}$ is called \emph{TV-achievable
(resp. KS-achievable) } for the \emph{universal} $(\mathcal{P}_{X^{n}},Q_{Y})$-simulation,
if there exists a sequence of simulators $\left\{ f_{n,k}\right\} _{k=1}^{\infty}$
such that 
\[
\limsup_{k\to\infty}\left|P_{Y_{n,k}}-Q_{Y}\right|_{\theta}\leq g(P_{X^{n}})
\]
for all $P_{X^{n}}\in\mathcal{P}_{X^{n}}$, where $P_{Y_{n,k}}:=P_{X^{n}}\circ f_{n,k}^{-1}$
and $\theta=\mathrm{TV}$ (resp. $\theta=\mathrm{KS}$). 
\end{defn}
\begin{defn}
\emph{The set of TV-achievable (resp. KS-achievable) functions} for
the \emph{universal} $(\mathcal{P}_{X^{n}},Q_{Y})$-simulation is
defined as 
\[
\mathcal{E}_{\theta}(\mathcal{P}_{X^{n}},Q_{Y}):=\left\{ g:\mathcal{P}_{X^{n}}\to\mathbb{R}\::g\textrm{ is }\theta\raisebox{.75pt}{-}\textrm{achievable}\right\} 
\]
where $\theta=\mathrm{TV}$ (resp. $\theta=\mathrm{KS}$). 
\end{defn}
According to Lebesgue's decomposition theorem \cite{dudley2002real},
the distributions of real-valued random variables can be partitioned
into three classes\footnote{We say a distribution (or a probability measure) $P$ is discrete
(or atomic) if it is purely atomic; continuous if it does not have
any atoms; discontinuous if it has at least one atom; absolutely continuous
if it is absolutely continuous with respect to Lebesgue measure, i.e.,
having a probability density function; and singular continuous if
it is continuous, and meanwhile, singular with respect to Lebesgue
measure.}: discontinuous distributions (including discrete distributions and
mixtures of discrete and continuous distributions), absolutely continuous
distributions, and continuous but not absolutely continuous distributions
(including singular continuous distributions and mixtures of singular
continuous and absolutely continuous distributions). The sets of these
distributions are respectively denoted as $\mathcal{P}_{\mathrm{dc}},\mathcal{P}_{\mathrm{ac}}$,
and $\mathcal{P}_{\mathrm{c}}\backslash\mathcal{P}_{\mathrm{ac}}$,
where $\mathcal{P}_{\mathrm{c}}=\mathcal{P}(\mathbb{R},\mathcal{B}_{\mathbb{R}})\backslash\mathcal{P}_{\mathrm{dc}}$
denotes the set of continuous distributions on $\left(\mathbb{R},\mathcal{B}_{\mathbb{R}}\right)$.

For the i.i.d. case, we define $\mathcal{P}_{X}^{(n)}:=\left\{ P_{X}^{n}:P_{X}\in\mathcal{P}_{X}\right\} $.
In this paper, we want to characterize $\mathcal{E}_{\theta}(\mathcal{P}_{X}^{(n)},Q_{Y})$
with $\mathcal{P}_{X}$ respectively set to $\mathcal{P}_{\mathrm{dc}},\mathcal{P}_{\mathrm{ac}}$,
or $\mathcal{P}_{\mathrm{c}}\backslash\mathcal{P}_{\mathrm{ac}}$.
Note that $\mathcal{E}_{\theta}(\mathcal{P}_{X}^{(n)},Q_{Y})$ is
an upper set. For brevity, for such sets and a function $g:\mathcal{P}_{X}\to\mathbb{R}$,
we denote $\mathcal{E}_{\theta}(\mathcal{P}_{X}^{(n)},Q_{Y})\preceq g$
if there exists a $g'\in\mathcal{E}_{\theta}(\mathcal{P}_{X}^{(n)},Q_{Y})$
such that $g'(P_{X})\leq g(P_{X}),\forall P_{X}\in\mathcal{P}_{X}$;
and $\mathcal{E}_{\theta}(\mathcal{P}_{X}^{(n)},Q_{Y})\succeq g$
if $g'(P_{X})\geq g(P_{X}),\forall P_{X}\in\mathcal{P}_{X}$ for all
$g'\in\mathcal{E}_{\theta}(\mathcal{P}_{X}^{(n)},Q_{Y})$. In addition,
$\mathcal{E}_{\theta}(\mathcal{P}_{X}^{(n)},Q_{Y})\asymp g$ if and
only if $\mathcal{E}_{\theta}(\mathcal{P}_{X}^{(n)},Q_{Y})\preceq g$
and $\mathcal{E}_{\theta}(\mathcal{P}_{X}^{(n)},Q_{Y})\succeq g$.
In general, there does not necessarily exist $g$ such that $\mathcal{E}_{\theta}(\mathcal{P}_{X}^{(n)},Q_{Y})\asymp g$.
However, if it exists, then $g(P_{X})=\inf_{g'\in\mathcal{E}_{\theta}(\mathcal{P}_{X}^{(n)},Q_{Y})}g'(P_{X})$
for $P_{X}\in\mathcal{P}_{X}$.

Similarly, we write $\mathcal{E}_{\theta}(\mathcal{P}_{X}^{(n)},Q_{Y})\dotpreceq g_{n}$
if there exists a $g_{n}'\in\mathcal{E}_{\theta}(\mathcal{P}_{X}^{(n)},Q_{Y})$
such that\footnote{Throughout this paper, for two positive sequences $f(n),g(n)$, we
write $f(n)\dotle g(n)$ or $g(n)\dotge f(n)$ if $\limsup_{n\to\infty}\frac{1}{n}\log\frac{f(n)}{g(n)}\le0$.
In addition, $f(n)\doteq g(n)$ if and only if $f(n)\dotle g(n)$
and $f(n)\dotge g(n)$. } $g_{n}'(P_{X})\dotle g_{n}(P_{X}),\forall P_{X}\in\mathcal{P}_{X}$;
and $\mathcal{E}_{\theta}(\mathcal{P}_{X}^{(n)},Q_{Y})\dotsucceq g$
if $g_{n}'(P_{X})\dotge g_{n}(P_{X}),\forall P_{X}\in\mathcal{P}_{X}$
for all $g_{n}'\in\mathcal{E}_{\theta}(\mathcal{P}_{X}^{(n)},Q_{Y})$.
In addition, $\mathcal{E}_{\theta}(\mathcal{P}_{X}^{(n)},Q_{Y})\dotasymp g_{n}$
if and only if $\mathcal{E}_{\theta}(\mathcal{P}_{X}^{(n)},Q_{Y})\dotpreceq g$
and $\mathcal{E}_{\theta}(\mathcal{P}_{X}^{(n)},Q_{Y})\dotsucceq g$.
Furthermore, $\mathcal{E}_{\theta}(\mathcal{P}_{X}^{(n)},Q_{Y})\asymp e^{-\Omega(g_{n})}$
(resp. $\mathcal{E}_{\theta}(\mathcal{P}_{X}^{(n)},Q_{Y})\asymp e^{-\omega(g_{n})}$)
if there exists a $g_{n}'\in\mathcal{E}_{\theta}(\mathcal{P}_{X}^{(n)},Q_{Y})$
such that\footnote{For two positive sequences $f(n),g(n)$, we write $f(n)=\Omega(g(n))$
(resp. $f(n)=\omega(g(n))$) if $\liminf_{n\to\infty}\frac{f(n)}{g(n)}>0$
(resp. $\lim_{n\to\infty}\frac{f(n)}{g(n)}=\infty$). } $g_{n}'(P_{X})=e^{-\Omega(g_{n}(P_{X}))}$ (resp. $g_{n}'(P_{X})=e^{-\omega(g_{n}(P_{X}))}$)
for all $P_{X}\in\mathcal{P}_{X}$.

Conversely, as illustrated in Fig. \ref{fig:simulation} (b), if $P_{X^{n}}$
is \emph{known}, we term such problems as \emph{(non-universal) $(P_{X^{n}},Q_{Y})$-simulation
problems}. The simulator $f:\mathbb{R}^{n}\to\mathbb{R}$ in the non-universal
simulation problem may depend on all of $P_{X^{n}},Q_{X}$, etc. 
\begin{defn}
The optimal \emph{TV-achievable (resp. KS-achievable)} approximation
error for the \emph{non-universal} $(P_{X^{n}},Q_{Y})$-simulation
is defined as 
\[
E_{\theta}(P_{X^{n}},Q_{Y}):=\inf_{f_{n}:\mathbb{R}^{n}\to\mathbb{R}}\left|P_{Y_{n}}-Q_{Y}\right|_{\theta},
\]
where $P_{Y_{n}}:=P_{X^{n}}\circ f_{n}^{-1}$ and $\theta=\mathrm{TV}$
(resp. $\theta=\mathrm{KS}$). 
\end{defn}
The non-universal $(P_{X^{n}},Q_{Y})$-simulation problem can be seen
as a special universal $(\mathcal{P}_{X^{n}},Q_{Y})$-simulation problem
with $\mathcal{P}_{X^{n}}$ set to $\left\{ P_{X^{n}}\right\} $.
Hence $\mathcal{E}_{\theta}(\left\{ P_{X^{n}}\right\} ,Q_{Y})\asymp E_{\theta}(P_{X^{n}},Q_{Y})$.
On the other hand, by definitions, the set $\mathcal{E}_{\theta}(\mathcal{P}_{X^{n}},Q_{Y})$
for the \emph{universal} $(\mathcal{P}_{X^{n}},Q_{Y})$-simulation
with $\theta\in\left\{ \mathrm{KS},\mathrm{TV}\right\} $ must satisfy
$\mathcal{E}_{\theta}(\mathcal{P}_{X^{n}},Q_{Y})\succeq E_{\theta}(P_{X^{n}},Q_{Y})$.
That is, the approximation errors for non-universal simulation problems
are not larger than those for universal simulation problems.

In general, simulating a continuous random variable is more difficult
than simulating a discontinuous one, as stated in the following lemma.
Hence in this paper, sometimes we only provide upper bounds on the
approximation errors for simulating continuous random variables. It
should be understood that those upper bounds are also upper bounds
for simulating any other random variables (e.g., discrete random variables).
Furthermore, to make our results easier to follow, we summarize them
in Table \ref{tab:Summary-of-our}.

\begin{table*}
\centering%
\begin{tabular}{|>{\centering}p{7cm}|c|}
\hline 
\multicolumn{2}{|c|}{\textbf{Simulating a Random Variable from a Stationary Memoryless
Process}}\tabularnewline
\hline 
\hline 
$\left(\mathcal{P}_{X},\mathcal{Q}_{Y}\right)$  & \emph{Non-universal Simulation}\tabularnewline
\hline 
(continuous, arbitrary)  & $E_{\theta}(P_{X},Q_{Y})=0$ for $\theta\in\left\{ \mathrm{KS},\mathrm{TV}\right\} $
(Prop. \ref{prop:NUcontinuous}) \tabularnewline
\hline 
(discontinuous, arbitrary)  & $E_{\mathrm{KS}}(P_{X}^{n},Q_{Y})\leq\frac{1}{2}\left(\max_{x}P_{X}(x)\right)^{n}$
(Cor. \ref{cor:NUdiscontinuous} \& Lem. \ref{lem:continuousisharder})\tabularnewline
\hline 
\textbf{Special Case 1:} (discontinuous, continuous)  & $E_{\mathrm{KS}}(P_{X}^{n},Q_{Y})=\frac{1}{2}\left(\max_{x}P_{X}(x)\right)^{n}$
(Cor. \ref{cor:NUdiscontinuous})\tabularnewline
\hline 
\textbf{Special Case 2:} (discrete with finite alphabet, discrete
with finite alphabet)  & $E_{\mathrm{KS}}(P_{X}^{n},Q_{Y})\dotle\left(\min_{x}P_{X}(x)\right)^{n}$
(Prop. \ref{prop:NUdiscrete})\tabularnewline
\hline 
$\left(\mathcal{P}_{X},\mathcal{Q}_{Y}\right)$  & \emph{Universal Simulation}\tabularnewline
\hline 
(absolutely continuous, arbitrary)  & $\mathcal{E}_{\theta}(\mathcal{P}_{X},Q_{Y})\asymp0$ for $\theta\in\left\{ \mathrm{KS},\mathrm{TV},\mathrm{Renyi}\right\} $
(Thm. \ref{thm:Ucontinuous} \& \ref{thm:Ucontinuous-1} )\tabularnewline
\hline 
(discontinuous, arbitrary)  & $\mathcal{E}_{\mathrm{KS}}(\mathcal{P}_{X}^{(n)},Q_{Y})\dotpreceq\left(\max_{x}P_{X}(x)\right)^{n}$
(Cor. \ref{cor:Udiscontinuous} \& Lem. \ref{lem:continuousisharder})\tabularnewline
\hline 
\textbf{Special Case:} (discontinuous, continuous)  & $\mathcal{E}_{\mathrm{KS}}(\mathcal{P}_{X}^{(n)},Q_{Y})\dotasymp\left(\max_{x}P_{X}(x)\right)^{n}$
(Cor. \ref{cor:Udiscontinuous})\tabularnewline
\hline 
(continuous but not absolutely continuous, arbitrary)  & $\mathcal{E}_{\mathrm{KS}}(\mathcal{P}_{X}^{(n)},Q_{Y})\asymp e^{-\omega\left(n\right)}$
(Cor. \ref{cor:Uc-ac})\tabularnewline
\hline 
\textbf{Special Case:} ($F_{X}$ is Hölder continuous with exponent
$\alpha$ where $0<\alpha\leq1$, arbitrary)  & $\mathcal{E}_{\mathrm{KS}}(\mathcal{P}_{X}^{(n)},Q_{Y})\asymp e^{-\alpha\Omega\left(n\log n\right)}$
(Cor. \ref{cor:Uc-ac})\tabularnewline
\hline 
\multicolumn{1}{>{\centering}p{7cm}}{} & \multicolumn{1}{c}{}\tabularnewline
\hline 
\multicolumn{2}{|c|}{\textbf{Simulating a Random Variable from a Markov Process with Order
$k$}}\tabularnewline
\hline 
\hline 
$\left(P_{X^{n}},\mathcal{Q}_{Y}\right)$  & \emph{Non-universal Simulation}\tabularnewline
\hline 
(a Markov chain of order $k$ with finite state space $\mathcal{X}$
and initial state $x_{-k+1}^{0}$, arbitrary)  & $E_{\mathrm{KS}}(P_{X^{n}},Q_{Y})\dotle e^{-nH_{\text{\ensuremath{\infty}}}(x_{-k+1}^{0},P_{X_{k+1}|X^{k}})}$
(Cor. \ref{cor:NUMarkov} \& Lem. \ref{lem:continuousisharder})\tabularnewline
\hline 
\textbf{Special Case:} (a Markov chain of order $k$ with finite state
space $\mathcal{X}$ and initial state $x_{-k+1}^{0}$, continuous)  & $E_{\mathrm{KS}}(P_{X^{n}},Q_{Y})\doteq e^{-nH_{\text{\ensuremath{\infty}}}(x_{-k+1}^{0},P_{X_{k+1}|X^{k}})}$
(Cor. \ref{cor:NUMarkov})\tabularnewline
\hline 
$\left(P_{X^{n}},\mathcal{Q}_{Y}\right)$  & \emph{Universal Simulation}\tabularnewline
\hline 
(a Markov chain of order $k$ with finite state space $\mathcal{X}$
and initial state $x_{-k+1}^{0}$, arbitrary)  & $\mathcal{E}_{\mathrm{KS}}(\mathcal{P}_{X_{k+1}|X^{k}}^{(n)},Q_{Y})\dotpreceq e^{-nH_{\text{\ensuremath{\infty}}}(x_{-k+1}^{0},P_{X_{k+1}|X^{k}})}$
(Thm. \ref{thm:UMarkov} \& Lem. \ref{lem:continuousisharder})\tabularnewline
\hline 
\textbf{Special Case:} (a Markov chain of order $k$ with finite state
space $\mathcal{X}$ and initial state $x_{-k+1}^{0}$, continuous)  & $\mathcal{E}_{\mathrm{KS}}(\mathcal{P}_{X_{k+1}|X^{k}}^{(n)},Q_{Y})\dotasymp e^{-nH_{\text{\ensuremath{\infty}}}(x_{-k+1}^{0},P_{X_{k+1}|X^{k}})}$
(Thm. \ref{thm:UMarkov})\tabularnewline
\hline 
\multicolumn{1}{>{\centering}p{7cm}}{} & \multicolumn{1}{c}{}\tabularnewline
\hline 
\multicolumn{2}{|c|}{\textbf{Simulating a Random Element from another Random Element }}\tabularnewline
\hline 
\hline 
$\left(\mathcal{P}_{X},\mathcal{Q}_{Y}\right)$  & \emph{Non-universal Simulation}\tabularnewline
\hline 
(continuous, arbitrary)  & $E_{\mathrm{TV}}(P_{X},Q_{Y})=0$ (Thm. \ref{thm:NUgeneralRV})\tabularnewline
\hline 
\textbf{Special Case:} (continuous random variable, arbitrary random
vector)  & $E_{\theta}(P_{X},Q_{Y})=0$ for $\theta\in\left\{ \mathrm{KS},\mathrm{TV}\right\} $
(Cor. \ref{cor:NUvector})\tabularnewline
\hline 
$\left(\mathcal{P}_{X},\mathcal{Q}_{Y}\right)$  & \emph{Universal Simulation }\tabularnewline
\hline 
(absolutely continuous respect to a continuous distribution, arbitrary)  & $\mathcal{E}_{\mathrm{TV}}(\mathcal{P}_{X},Q_{Y})\asymp0$ (Thm. \ref{thm:UgeneralRV})\tabularnewline
\hline 
\textbf{Special Case: }(absolutely continuous random variable, arbitrary
random vector)  & $\mathcal{E}_{\theta}(\mathcal{P}_{X},Q_{Y})\asymp0$ for $\theta\in\left\{ \mathrm{KS},\mathrm{TV}\right\} $
(Cor. \ref{cor:Uvector})\tabularnewline
\hline 
\end{tabular}\caption{\label{tab:Summary-of-our}Summary of our results. Here $\mathcal{P}_{X}$
and $\mathcal{Q}_{Y}$ respectively denote the classes that $P_{X}$ and $Q_{Y}$
 belong to.}
\end{table*}
\begin{lem}
\label{lem:continuousisharder} Assume $Q_{Y}$ and $Q_{Z}$ are two
distributions defined on $\left(\mathbb{R},\mathcal{B}_{\mathbb{R}}\right)$,
and moreover, $Q_{Y}$ is continuous. Then the approximation errors
for non-universal and universal simulations satisfy 
\begin{align}
E_{\theta}(P_{X^{n}},Q_{Z}) & \leq E_{\theta}(P_{X^{n}},Q_{Y}),\label{eq:-25}\\
\mathcal{E}_{\theta}(\mathcal{P}_{X^{n}},Q_{Z}) & \supseteq\mathcal{E}_{\theta}(\mathcal{P}_{X^{n}},Q_{Y}),\label{eq:-28}
\end{align}
for any $P_{X^{n}}$ and $\mathcal{P}_{X^{n}}$, where $\theta\in\left\{ \mathrm{KS},\mathrm{TV}\right\} $. 
\end{lem}
\begin{proof}
Proposition \ref{prop:NUcontinuous} (which is given in the next section)
states that there exists a non-decreasing mapping $z=g(y):\:\mathbb{R}\to\mathbb{R}$
such that $Z=g(Y)\sim Q_{Z}$, where $Y\sim Q_{Y}$. Observe that
for any $P_{Z}$,
\begin{align*}
\left|P_{Z}-Q_{Z}\right|_{\mathrm{KS}} & =\sup_{A\in\mathcal{I}}\left|P_{Z}(A)-Q_{Z}(A)\right|\\
 & =\sup_{A\in\mathcal{I}}\left|P_{Y}(g^{-1}(A))-Q_{Y}(g^{-1}(A))\right|\\
 & \leq\sup_{B\in\mathcal{I}}\left|P_{Y}(B)-Q_{Y}(B)\right|\\
 & =\left|P_{Y}-Q_{Y}\right|_{\mathrm{KS}},
\end{align*}
where $\mathcal{I}:=\left\{ (-\infty,y]:\:y\in\mathbb{R}\right\} $.
Hence \eqref{eq:-25} and \eqref{eq:-28} hold for $\theta=\mathrm{KS}$.

By similar steps but with $\mathcal{I}$ replaced by $\mathcal{B}_{\mathbb{R}}$,
we can easily obtain that \eqref{eq:-25} and \eqref{eq:-28} also
hold for $\theta=\mathrm{TV}$. 
\end{proof}

\section{Non-universal Simulation from a Stationary Memoryless Process }

In this section, we consider non-universal simulation of a real-valued
random variable. If the seed distribution $P_{X}$ is continuous and
the target distribution $Q_{Y}$ is arbitrary, then we can simulate
a random variable $Y$ that exactly follows the distribution $Q_{Y}$.
 The following is a well-known result for such a case. Hence the
proof is omitted.
\begin{prop}
\label{prop:NUcontinuous}For a continuous distribution $P_{X}$ and
an arbitrary distribution $Q_{Y}$, using the inverse transform sampling
function $y=G_{Y}^{-1}\left(F_{X}(x)\right)$, we obtain $P_{Y}=Q_{Y}$,
where\footnote{Here the minimum exists since CDFs are right-continuous.}
$G_{Y}^{-1}\left(t\right):=\min\left\{ y:G_{Y}(y)\geq t\right\} $
denotes the quantile function (generalized inverse distribution function)
of $G_{Y}$. That is, $E_{\theta}(P_{X},Q_{Y})=0$ for $\theta\in\left\{ \mathrm{KS},\mathrm{TV}\right\} $.
\end{prop}
Next we consider the case $P_{X}$ is discontinuous. For this case,
exact simulation cannot be obtained. 
\begin{prop}
\label{prop:NUdiscontinuous} Assume $P_{X}$ is discontinuous and
$Q_{Y}$ is continuous. Then for the non-universal $(P_{X},Q_{Y})$-simulation
problem,\footnote{For simplicity, we denote $P_{X}(\{x\})$ for $x\in\mathbb{R}$ as
$P_{X}(x)$. Hence for a discrete random variable $X$, $P_{X}(x)$
is the probability mass function of $X$. } $E_{\mathrm{KS}}(P_{X},Q_{Y})=\frac{1}{2}\max_{x}P_{X}(x)$. 
\end{prop}
\begin{proof}
Denote $A$ as the set of discontinuity points of $F_{X}$. Then for
each $x\in\mathbb{R}\backslash A$, map $x$ to $G_{Y}^{-1}(F_{X}(x))$.
For each $x\in A$, map $x$ to $G_{Y}^{-1}(\lim_{\tilde{x}\uparrow x}F_{X}(\tilde{x}))+\frac{1}{2}P_{X}(x))$.
For such mapping, we have $\left|P_{Y}-Q_{Y}\right|_{\mathrm{KS}}=\frac{1}{2}\max_{x}P_{X}(x)$.
Furthermore, the converse is obvious. 
\end{proof}
Applying Proposition \ref{prop:NUdiscontinuous} to the vector case,
we get the following corollary. 
\begin{cor}
\label{cor:NUdiscontinuous}Assume $P_{X}$ is discontinuous and $Q_{Y}$
is continuous. Then for the non-universal $(P_{X}^{n},Q_{Y})$-simulation
problem, $E_{\mathrm{KS}}(P_{X}^{n},Q_{Y})=\frac{1}{2}\left(\max_{x}P_{X}(x)\right)^{n}$. 
\end{cor}
The result above shows that if the seed and target distributions are
respectively discontinuous and continuous, then the optimal approximation
error vanishes exponentially fast. We next show that if the seed and
target distributions are both discrete, the optimal approximation
error vanishes faster. The proof of Proposition \ref{prop:NUdiscrete}
is provided in Appendix \ref{sec:Proof-of-Theorem-discrete}. 
\begin{prop}
\label{prop:NUdiscrete} Assume both $P_{X}$ and $Q_{Y}$ are discrete
with finite alphabets $\mathcal{X}$ and $\mathcal{Y}$ respectively.
Then for the non-universal $(P_{X}^{n},Q_{Y})$-simulation problem,
$E_{\mathrm{KS}}(P_{X}^{n},Q_{Y})\dotle\left(\min_{x}P_{X}(x)\right)^{n}$. 
\end{prop}
\begin{rem}
More specifically, we can prove that for $n\geq|\mathcal{Y}|\max_{1\le i\leq|\mathcal{X}|-1}\frac{P_{X}(x_{i})}{P_{X}(x_{i+1})}$,
\begin{align*}
E_{\mathrm{KS}}(P_{X}^{n},Q_{Y}) & \leq\frac{1}{2}P_{X}(x_{|\mathcal{X}|-1})\left(P_{X}(x_{|\mathcal{X}|})\right)^{n-1},
\end{align*}
where $x_{1},x_{2},...,x_{|\mathcal{X}|}$ is a resulting sequence
after sorting the elements in $\mathcal{X}$ such that $P_{X}(x_{1})\geq P_{X}(x_{2})\geq...\geq P_{X}(x_{|\mathcal{X}|})$. 
\end{rem}

\section{Universal Simulation from a Stationary Memoryless Process }

In this section, we consider universal simulation of a real-valued
random variable. For universal simulation, we divide the seed distributions
into three kinds: absolutely continuous, discontinuous, as well as
continuous but not absolutely continuous distributions.

\subsection{Absolutely continuous seed distributions}

We first consider absolutely continuous seed distributions, and show
an impossibility result for this case. 
\begin{prop}
\label{prop:Uconverse} For non-degenerate distribution $Q_{Y}$,
there is no simulator (measurable function) $Y=f(X):\left(\mathbb{R},\mathcal{B}_{\mathbb{R}}\right)\to\left(\mathbb{R},\mathcal{B}_{\mathbb{R}}\right)$
such that $P_{Y}=Q_{Y}$ for any absolutely continuous $P_{X}$. 
\end{prop}
\begin{proof}
Suppose that there exists a measurable function $f:\left(\mathbb{R},\mathcal{B}_{\mathbb{R}}\right)\to\left(\mathbb{R},\mathcal{B}_{\mathbb{R}}\right)$
such that $P_{Y}=Q_{Y}$ for any absolutely continuous $P_{X}$.

Case 1: Suppose that there exists a set $A\in\mathcal{B}_{\mathbb{R}}$
such that both $f^{-1}(A)$ and $\mathbb{R}\backslash f^{-1}(A)$
have positive Lebesgue measures. Then $P_{Y}(A)=P_{X}(f^{-1}(A))$.
For two absolutely continuous measures $P_{X}$ and $\widetilde{P}_{X}$
such that $P_{X}(f^{-1}(A))\neq\widetilde{P}_{X}(f^{-1}(A))$, then
$P_{Y}(A)\neq\widetilde{P}_{Y}(A)$, where $\widetilde{P}_{Y}$ is
the distribution induced by $\widetilde{P}_{X}$ through the mapping
$f$. This implies that $P_{Y}(A)\neq Q_{Y}(A)$ or $\widetilde{P}_{Y}(A)\neq Q_{Y}(A)$.
This contradicts with the assumption that $P_{Y}=Q_{Y}$  for any
absolutely continuous $P_{X}$.

Case 2: Suppose that either $f^{-1}(A)$ or $\mathbb{R}\backslash f^{-1}(A)$
has zero Lebesgue measure for all $A\in\mathcal{B}_{\mathbb{R}}$.
Then for any absolutely continuous $P_{X}$, we have $P_{Y}(A)=P_{X}(f^{-1}(A))=0$
or $P_{Y}(\mathbb{R}\backslash A)=1-P_{Y}(A)=1-P_{X}(f^{-1}(A))=P_{X}(f^{-1}(\mathbb{R}\backslash A))=0$
for all $A\in\mathcal{B}_{\mathbb{R}}$. That is, $P_{Y}(A)=0$ or
$1$ for all $A\in\mathcal{B}_{\mathbb{R}}$. However for any non-degenerate
measure $Q_{Y}$, there exists an $A\in\mathcal{B}_{\mathbb{R}}$
such that $0<Q_{Y}(A)<1$. This contradicts with the assumption that
$P_{Y}=Q_{Y}$.

Combining the two cases above, we have Proposition \ref{prop:Uconverse}.
\end{proof}
The theorem above implies that for any simulator $f:\mathbb{R}\to\mathbb{R}$,
we always have $\left|P_{Y}-Q_{Y}\right|_{\mathrm{TV}}>0$ for some
absolutely continuous $P_{X}$. However, we can prove that there exists
a sequence of simulators that make the TV-approximation error $\left|P_{Y}-Q_{Y}\right|_{\mathrm{TV}}$
arbitrarily close to zero for any absolutely continuous $P_{X}$. 
\begin{thm}
\label{thm:Ucontinuous} Assume $\mathcal{P}_{X}=\mathcal{P}_{\mathrm{ac}}$
and $Q_{Y}$ is arbitrary. Then for the universal $(\mathcal{P}_{X},Q_{Y})$-simulation
problem, $\mathcal{E}_{\theta}(\mathcal{P}_{X},Q_{Y})\asymp0$ for
$\theta\in\left\{ \mathrm{KS},\mathrm{TV}\right\} $. 
\end{thm}
Given Proposition \ref{prop:Uconverse}, I think this result is rather
surprising and counter-intuitive. If $f$ is a differentiable bijective
function, then the input distribution is determined by $f$ and the
output distribution, since $p_{X}(x)=p_{Y}(f(x))f'(x)$, where $p_{X}$
and $p_{Y}$ are respectively the PDFs (probability density functions,
i.e., Radon\textendash Nikodym derivatives respect to the Lebesgue
measure) of $P_{X}$ and $P_{Y}$. Hence given $f$ and the output
distribution, the input distribution is unique. However, in our case,
we consider a sequence of non-bijective mappings $f_{n}$. Hence given
$f_{n}$ and the output distribution, the input distribution is not
unique. The essence of our proof of this theorem is that any PDF $p_{X}$
can be approximated within any level of approximation error by a sequence
of step functions, and on the other hand, such step functions can
be used to generate any distribution in a universal way. Hence $P_{X}$
can be used to simulate any distribution within any level of approximation
error. 
\begin{proof}
We first restrict our attention to the case that $Q_{Y}$ is absolutely
continuous. Let $p_{X}$, $p_{Y}$, and $q_{Y}$ be the PDFs of $P_{X}$,
$P_{Y}$, and $Q_{Y}$ respectively.

\textbf{Universal Mapping: }Partition the real line into intervals
with the same length $\Delta$, i.e., $\bigcup_{i=-\infty}^{\infty}(i\Delta,(i+1)\Delta]$.
We first simulate a uniform distribution on $[a,b]$ by mapping each
interval $(i\Delta,(i+1)\Delta]$ into $[a,b]$ using the linear function
$x\mapsto a+\frac{b-a}{\Delta}(x-i\Delta)$. We then transform the
output distribution to the target distribution $Q_{Y}$, by using
function $x\mapsto G_{Y}^{-1}\left(\frac{x-a}{b-a}\right)$, where
$G_{Y}^{-1}\left(t\right):=\min\left\{ y:G_{Y}(y)\geq t\right\} $.
Therefore, each $x\in(i\Delta,(i+1)\Delta]$ is mapped to $G_{Y}^{-1}\left(\frac{1}{\Delta}(x-i\Delta)\right)$.
Hence the final mapping is 
\[
f_{\Delta}(x):=\sum_{i=-\infty}^{\infty}G_{Y}^{-1}\left(\frac{1}{\Delta}(x-i\Delta)\right)1\left\{ x\in(i\Delta,(i+1)\Delta]\right\} ,
\]
which is shown in Fig. \ref{fig:Illustration-of-the}. Furthermore,
some properties on squeezing such a periodic function are provided
in Appendix \ref{sec:Proof-of-Theorem-Squeeze}.

\begin{figure}
\centering\begin{tikzpicture}[domain=-4:4]
\draw[thin,color=gray] (-0.1,-1.1) ;
\draw[->] (-4,0) -- (4,0) node[right] {$x$};
\draw[->] (0,-3) -- (0,3) node[above] {$f(x)$};
\draw[loosely dotted] (-3,-3) -- (-3,3);
\draw[loosely dotted] (-2,-3) -- (-2,3);
\draw[loosely dotted] (-1,-3) -- (-1,3);
\draw[loosely dotted] (1,-3) -- (1,3);
\draw[loosely dotted] (3,-3) -- (3,3);
\draw[loosely dotted] (2,-3) -- (2,3);
\draw[scale=1,domain=-3.999:-3.001,smooth,variable=\x,blue] plot ({\x},{-0.4*ln(1/(\x+4)-1))});
\draw[scale=1,domain=-2.999:-2.001,smooth,variable=\x,blue] plot ({\x},{-0.4*ln(1/(\x+3)-1))});
\draw[scale=1,domain=-1.999:-1.001,smooth,variable=\x,blue] plot ({\x},{-0.4*ln(1/(\x+2)-1))});
\draw[scale=1,domain=-0.999:-0.001,smooth,variable=\x,blue] plot ({\x},{-0.4*ln(1/(\x+1)-1))});
\draw[scale=1,domain=0.001:0.999,smooth,variable=\x,blue] plot ({\x},{-0.4*ln(1/(\x)-1))});
\draw[scale=1,domain=1.001:1.999,smooth,variable=\x,blue] plot ({\x},{-0.4*ln(1/(\x-1)-1))});
\draw[scale=1,domain=2.001:2.999,smooth,variable=\x,blue] plot ({\x},{-0.4*ln(1/(\x-2)-1))});
\draw[scale=1,domain=3.001:3.999,smooth,variable=\x,blue] plot ({\x},{-0.4*ln(1/(\x-3)-1))});
 \node [below=0cm, align=flush center,text width=8cm] at (-3.8,0)         {            $-4\Delta$         };
 \node [below=0cm, align=flush center,text width=8cm] at (-2.8,0)         {            $-3\Delta$         };
 \node [below=0cm, align=flush center,text width=8cm] at (-1.8,0)         {            $-2\Delta$         };
 \node [below=0cm, align=flush center,text width=8cm] at (-0.8,0)         {            $-\Delta$         };
 \node [below=0cm, align=flush center,text width=8cm] at (0.2,0)         {            $0$         };
 \node [below=0cm, align=flush center,text width=8cm] at (1.2,0)         {            $\Delta$         };
 \node [below=0cm, align=flush center,text width=8cm] at (2.2,0)         {            $2\Delta$         };
 \node [below=0cm, align=flush center,text width=8cm] at (3.2,0)         {            $3\Delta$         };
 \node [below=0cm, align=flush center,text width=8cm] at (4.2,0)         {            $4\Delta$         };
\end{tikzpicture}\caption{\label{fig:Illustration-of-the}Illustration of the universal mapping.}
\end{figure}
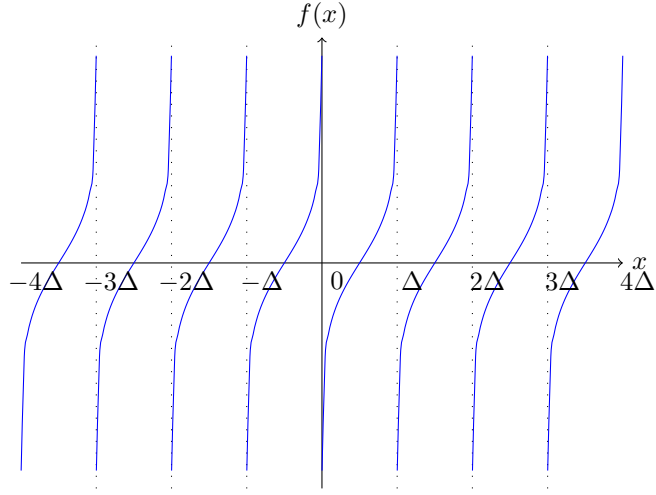

The PDF of the output of this mapping (respect to the input distribution
$P_{X}$) is denoted as $P_{Y}$. However, if the input distribution
is $\widehat{P}_{X}$ with PDF $\widehat{p}(x):=\sum_{i=-\infty}^{\infty}\frac{1}{\Delta}\int_{i\Delta}^{(i+1)\Delta}p(x)dx\cdot1\left\{ x\in(i\Delta,(i+1)\Delta]\right\} $,
then the output distribution becomes the one with CDF $\widehat{G}_{Y}(y)$
satisfying 
\begin{align}
\widehat{G}_{Y}(y) & =\int_{\left\{ x:f_{\Delta}(x)\le y\right\} }\sum_{i=-\infty}^{\infty}\frac{1}{\Delta}\int_{i\Delta}^{(i+1)\Delta}p(x)dx\cdot1\left\{ x\in(i\Delta,(i+1)\Delta]\right\} dx\nonumber \\
 & =\sum_{i=-\infty}^{\infty}\int_{i\Delta}^{i\Delta+\Delta G_{Y}(y)}dx\frac{1}{\Delta}\int_{i\Delta}^{(i+1)\Delta}p(x)dx\label{eq:-3}\\
 & =\sum_{i=-\infty}^{\infty}\Delta G_{Y}(y)\frac{1}{\Delta}\int_{i\Delta}^{(i+1)\Delta}p(x)dx\nonumber \\
 & =G_{Y}(y),\nonumber 
\end{align}
where in \eqref{eq:-3}, swapping the integral and the sum follows
from Fubini's theorem. Hence the output distribution induced by inputting
$\widehat{P}_{X}$ results in $Q_{Y}$ exactly. 

Based on the above observations, we have 
\begin{align}
\left|P_{Y}-Q_{Y}\right|_{\mathrm{TV}} & \leq\left|P_{X}-\widehat{P}_{X}\right|_{\mathrm{TV}}\label{eq:-1-2}\\
 & =\int_{-\infty}^{\infty}\left|p(x)-\widehat{p}(x)\right|dx\nonumber \\
 & =2\int_{-\infty}^{\infty}\left[p(x)-\widehat{p}(x)\right]^{+}dx,\label{eq:-6}
\end{align}
where $\left[z\right]^{+}=\max\{z,0\}$, and \eqref{eq:-1-2} follows
from the data processing inequality on the total variable distance,
i.e., $\left|P_{Y}-Q_{Y}\right|_{\mathrm{TV}}\leq\left|P_{X}-Q_{X}\right|_{\mathrm{TV}}$
with $P_{Y}\left(\cdot\right):=\int P_{Y|X}(\cdot|x)dP_{X}(x)$ and
$Q_{Y}\left(\cdot\right):=\int P_{Y|X}(\cdot|x)dQ_{X}(x)$.

Observe that $\left[p(x)-\widehat{p}(x)\right]^{+}\leq p(x)$ and
$p(x)$ is integrable, and moreover,
\begin{align}
\lim_{\Delta\to0}\widehat{p}(x) & =\lim_{\Delta\to0}\sum_{i=-\infty}^{\infty}\frac{1}{\Delta}\int_{i\Delta}^{(i+1)\Delta}p(x')dx'1\left\{ x\in(i\Delta,(i+1)\Delta]\right\} \nonumber \\
 & =\lim_{\Delta\to0}\frac{1}{\Delta}\int_{\left\lfloor \frac{x}{\Delta}\right\rfloor \Delta}^{\left(\left\lfloor \frac{x}{\Delta}\right\rfloor +1\right)\Delta}p(x')dx'\nonumber \\
 & =p(x)\;\mathrm{a.e.},\label{eq:-9-2}
\end{align}
where \eqref{eq:-9-2} follows by Lebesgue's differentiation theorem
\cite[Thm. 7.7]{rudin2006real}. Hence by Lebesgue's dominated convergence
theorem \cite[Thm. 1.34]{rudin2006real}, 
\begin{align}
\lim_{\Delta\to0}\int_{-\infty}^{\infty}\left[p(x)-\widehat{p}(x)\right]^{+}dx & =\int_{-\infty}^{\infty}\lim_{\Delta\to0}\left[p(x)-\widehat{p}(x)\right]^{+}dx\nonumber \\
 & =0.\label{eq:-7}
\end{align}

Therefore, combining \eqref{eq:-6} with \eqref{eq:-7} yields
\[
\lim_{\Delta\to0}\left|P_{Y}-Q_{Y}\right|_{\mathrm{TV}}=0.
\]
That is, $\mathcal{E}_{\mathrm{TV}}(\mathcal{P}_{X},Q_{Y})\asymp0.$

If $Q_{Y}$ is not absolutely continuous, we can first simulate an
absolutely continuous random variable $Z$ with distribution $Q_{Z}$
and then use it to simulate $Y\sim Q_{Y}$. As stated in Lemma \ref{lem:continuousisharder},
this will result in a smaller TV-approximation error for $(P_{X},Q_{Y})$-simulation
problem than that for $(P_{X},Q_{Z})$-simulation problem. Hence the
TV-approximation error for this case also approaches to zero as $\Delta\to0$.
\end{proof}
For the universal mapping proposed in the proof of Theorem \ref{thm:Ucontinuous},
the induced approximation error $\left|P_{Y}-Q_{Y}\right|_{\mathrm{TV}}$
depends on the interval length $\Delta$, and converges to zero as
$\Delta\to0$. We next investigate how fast the approximation error
converges to zero as $\Delta\to0$. 
\begin{prop}[Convergence Rate as $\Delta\to0$]
\label{prop:Ucontinuous--1} Assume $P_{X}$ is an absolutely continuous
distribution with an a.e. (almost everywhere) continuously differentiable
PDF $p_{X}$ such that $\left|p_{X}'(x)\right|$ is bounded, and $Q_{Y}$
is an arbitrary distribution. Then the TV-approximation error induced
by the universal mapping $x\mapsto\sum_{i=-\infty}^{\infty}G_{Y}^{-1}\left(\frac{1}{\Delta}(x-i\Delta)\right)1\left\{ x\in(i\Delta,(i+1)\Delta]\right\} $
satisfies $\limsup_{\Delta\to0}\frac{1}{\Delta}\left|P_{Y}-Q_{Y}\right|_{\mathrm{TV}}\leq\int\left|p_{X}'(x)\right|dx$,
where $P_{\Delta}$ denotes the distribution of the output $Y$. 
\end{prop}
\begin{proof}
The mean value theorem implies there exists $x_{i}\in(i\Delta,(i+1)\Delta]$
such that $P_{X}(x_{i})=\frac{1}{\Delta}\int_{i\Delta}^{(i+1)\Delta}p_{X}(x)dx$.
By Taylor's theorem, we have $p_{X}(x)=p_{X}(x_{i})+p_{X}'(x_{i,\Delta})(x-x_{i})$
for some $x_{i,\Delta}\in(i\Delta,(i+1)\Delta]$. Therefore, from
Remark \ref{rem:More-specifically,-it}, we have 
\begin{align*}
\frac{1}{\Delta}\left|P_{Y}-Q_{Y}\right|_{\mathrm{TV}} & \leq\frac{1}{\Delta}\sum_{i=-\infty}^{\infty}\int_{i\Delta}^{(i+1)\Delta}\left|p_{X}(x)-\frac{1}{\Delta}\int_{i\Delta}^{(i+1)\Delta}p_{X}(x)dx\right|dx\\
 & =\frac{1}{\Delta}\sum_{i=-\infty}^{\infty}\int_{i\Delta}^{(i+1)\Delta}\left|p_{X}(x_{i})+p_{X}'(x_{i,\Delta})(x-x_{i})-p_{X}(x_{i})\right|dx\\
 & \leq\sum_{i=-\infty}^{\infty}\left|p_{X}'(x_{i,\Delta})\right|\Delta.
\end{align*}
Since $\left|p_{X}'(x)\right|$ is continuous a.e. and bounded, $\left|p_{X}'(x)\right|$
is Riemann-integrable on every interval $[a,b]$ with $a<b$. Then
\begin{align*}
\limsup_{\Delta\to0}\sum_{i=-\infty}^{\infty}\left|p_{X}'(x_{i,\Delta})\right|\Delta & =\mathsf{Riemann}\raisebox{.75pt}{-}\int\left|p_{X}'(x)\right|dx,
\end{align*}
where $\mathsf{Riemann}\raisebox{.75pt}{-}\int$ denotes the Riemann-integral.
Furthermore, for any non-negative Riemann-integrable function, its
Riemann-integral and Lebesgue-integral are the same. That is 
\begin{align*}
\mathsf{Riemann}\raisebox{.75pt}{-}\int\left|p_{X}'(x)\right|dx & =\int\left|p_{X}'(x)\right|dx.
\end{align*}
Therefore, $\limsup_{\Delta\to0}\frac{1}{\Delta}\left|P_{Y}-Q_{Y}\right|_{\mathrm{TV}}\leq\int\left|p_{X}'(x)\right|dx$. 
\end{proof}
Proposition \ref{prop:Ucontinuous--1} implies that if the total variation
$\int\left|p_{X}'(x)\right|dx$ of $p_{X}$ is finite, then the approximation
error $\left|P_{Y}-Q_{Y}\right|_{\mathrm{TV}}$ converges to zero
at least linearly fast as $\Delta\to0$. If $\int\left|p_{X}'(x)\right|dx$
is infinity, then it is not easy to obtain a general bound on $\left|P_{Y}-Q_{Y}\right|_{\mathrm{TV}}$.
However, for some special cases, e.g., $p_{X}=-\log x,x\in(0,1]$
or $p_{X}=\left(1-r\right)x^{-r},x\in(0,1],r\in(0,1)$, we provide
upper bounds as follows. 
\begin{example}[Convergence Rate as $\Delta\to0$ for $\int\left|p_{X}'(x)\right|dx=\infty$]
\label{exam:U} \cite{owens2014exploring} For the PDF $p_{X}(x)=-\log x,x\in(0,1]$,
$\left|P_{Y}-Q_{Y}\right|_{\mathrm{TV}}\leq\frac{\Delta}{2}\ln2\pi\frac{1}{\Delta}$.
For the PDF $p_{X}(x)=\left(1-r\right)x^{-r},x\in(0,1],r\in(0,1)$,
$\left|P_{Y}-Q_{Y}\right|_{\mathrm{TV}}\leq C\Delta^{1-r}$ for some
constant $C$. 
\end{example}
Universal simulations in Theorem \ref{thm:Ucontinuous} for the uniform
distribution on $[0,1]$ for different $P_{X}$ are illustrated in
Fig. \ref{fig:blind}. For Gaussian and exponential distributions,
the total variations of their PDFs are finite. Hence the approximation
errors for these two distributions decay linearly in $\Delta$. For
logarithmic and polynomial-like distributions, the total variations
of their PDFs are infinite. As stated in Proposition \ref{exam:U},
the approximation errors for these two distributions decay respectively
in order of $\Delta\ln\frac{1}{\Delta}$ and $\Delta^{1-r}$.

\begin{figure}[t]
\centering \includegraphics[width=0.24\columnwidth]{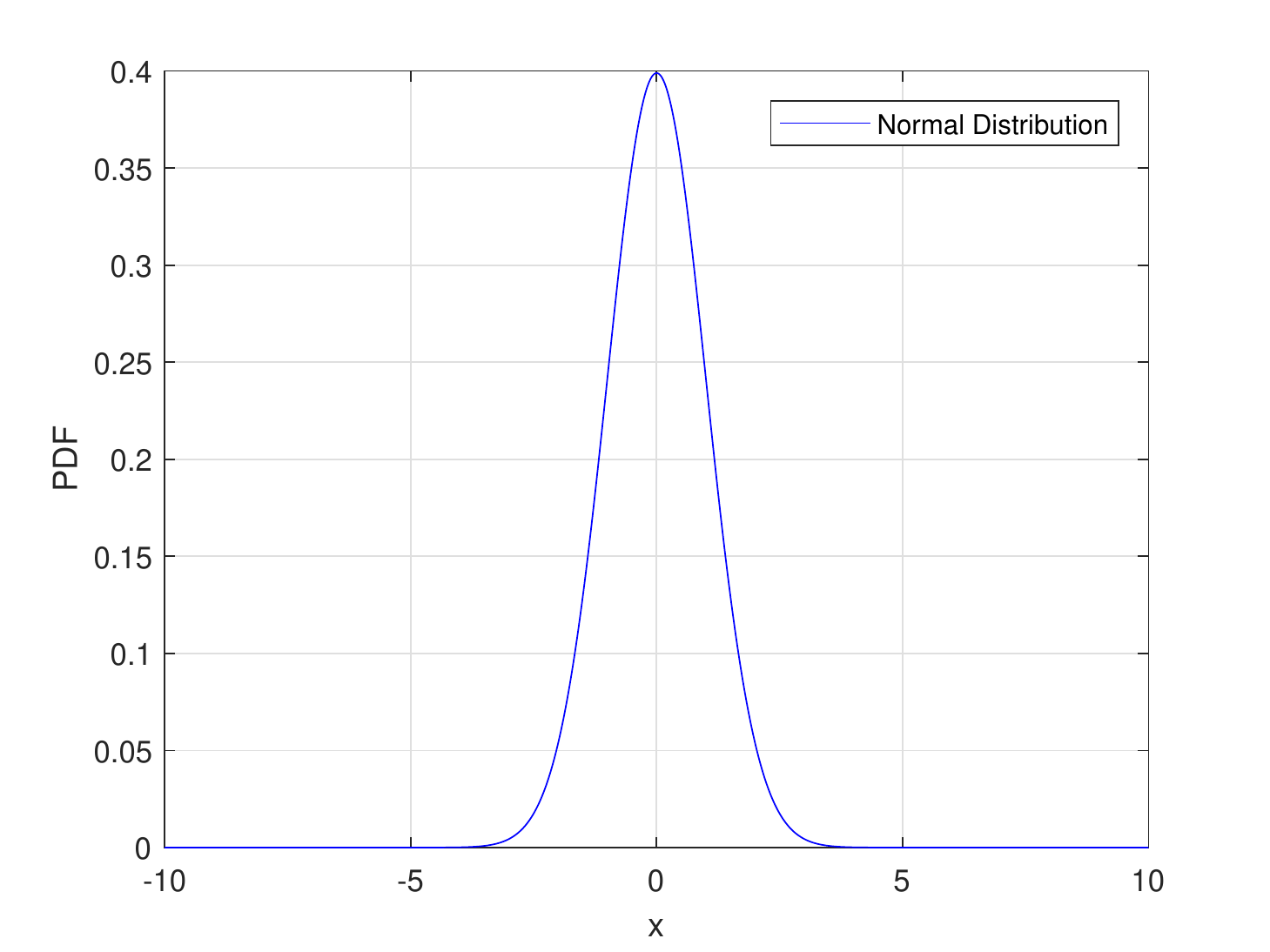} \includegraphics[width=0.24\columnwidth]{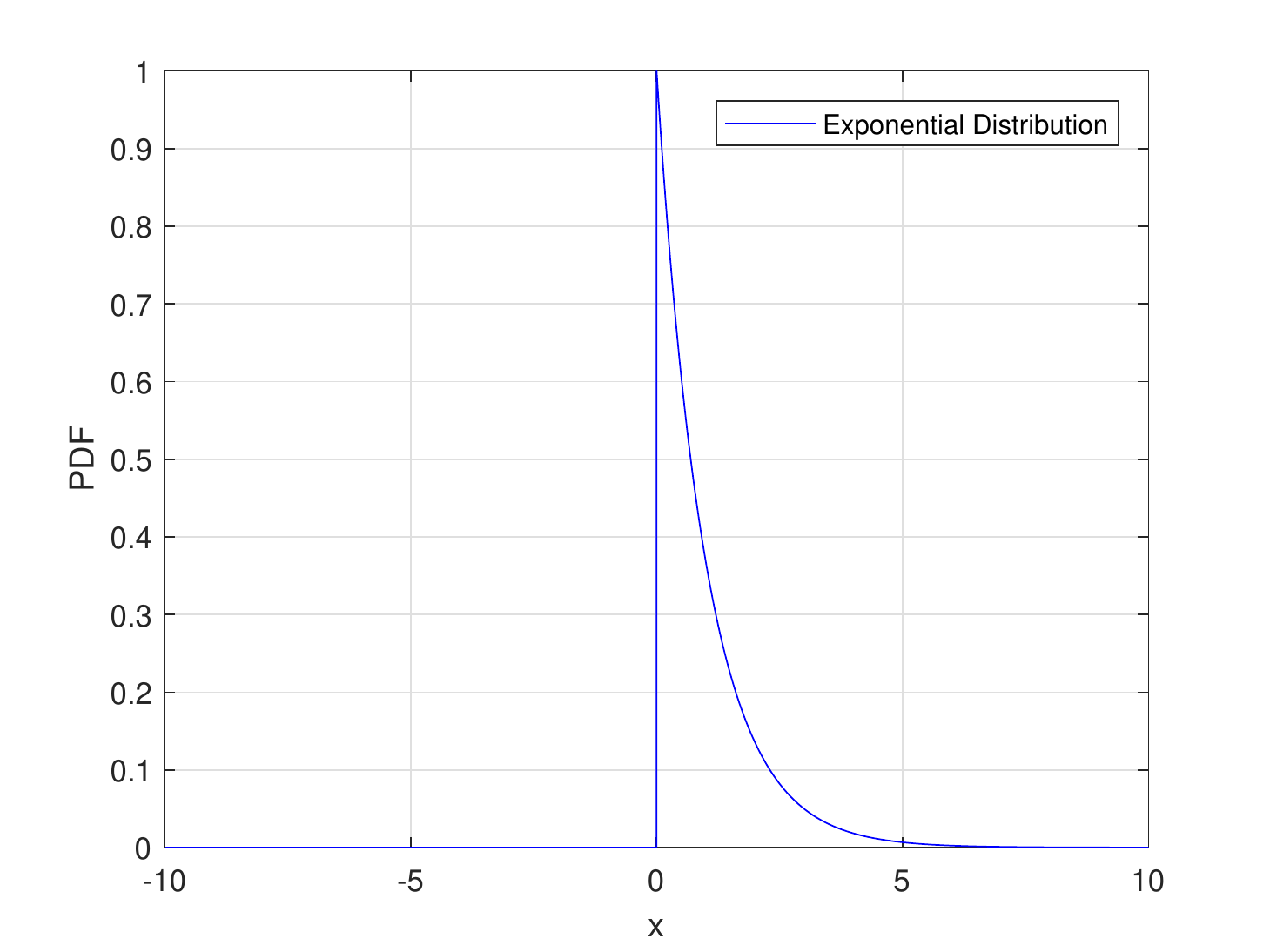}
\includegraphics[width=0.24\columnwidth]{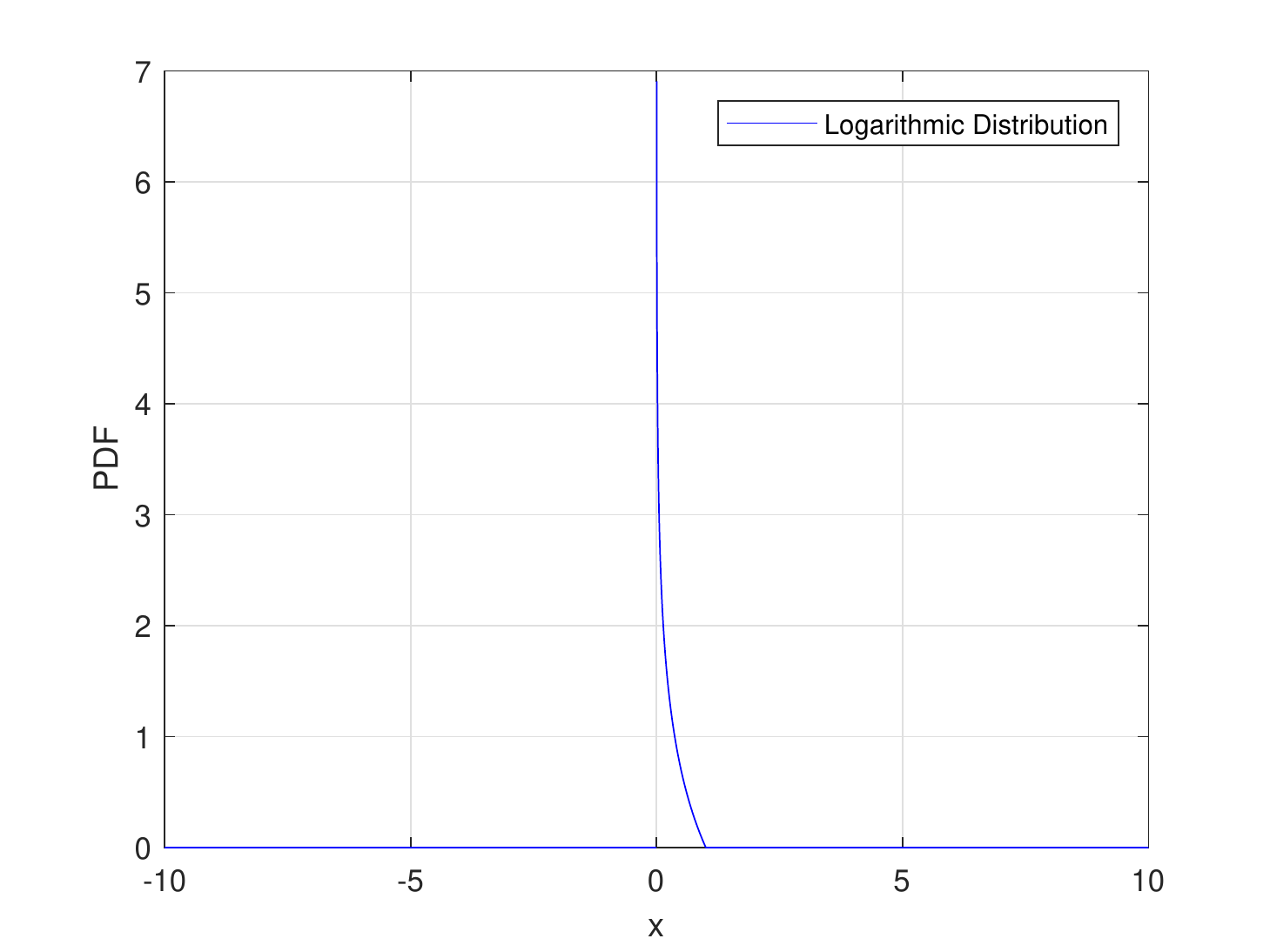} \includegraphics[width=0.24\columnwidth]{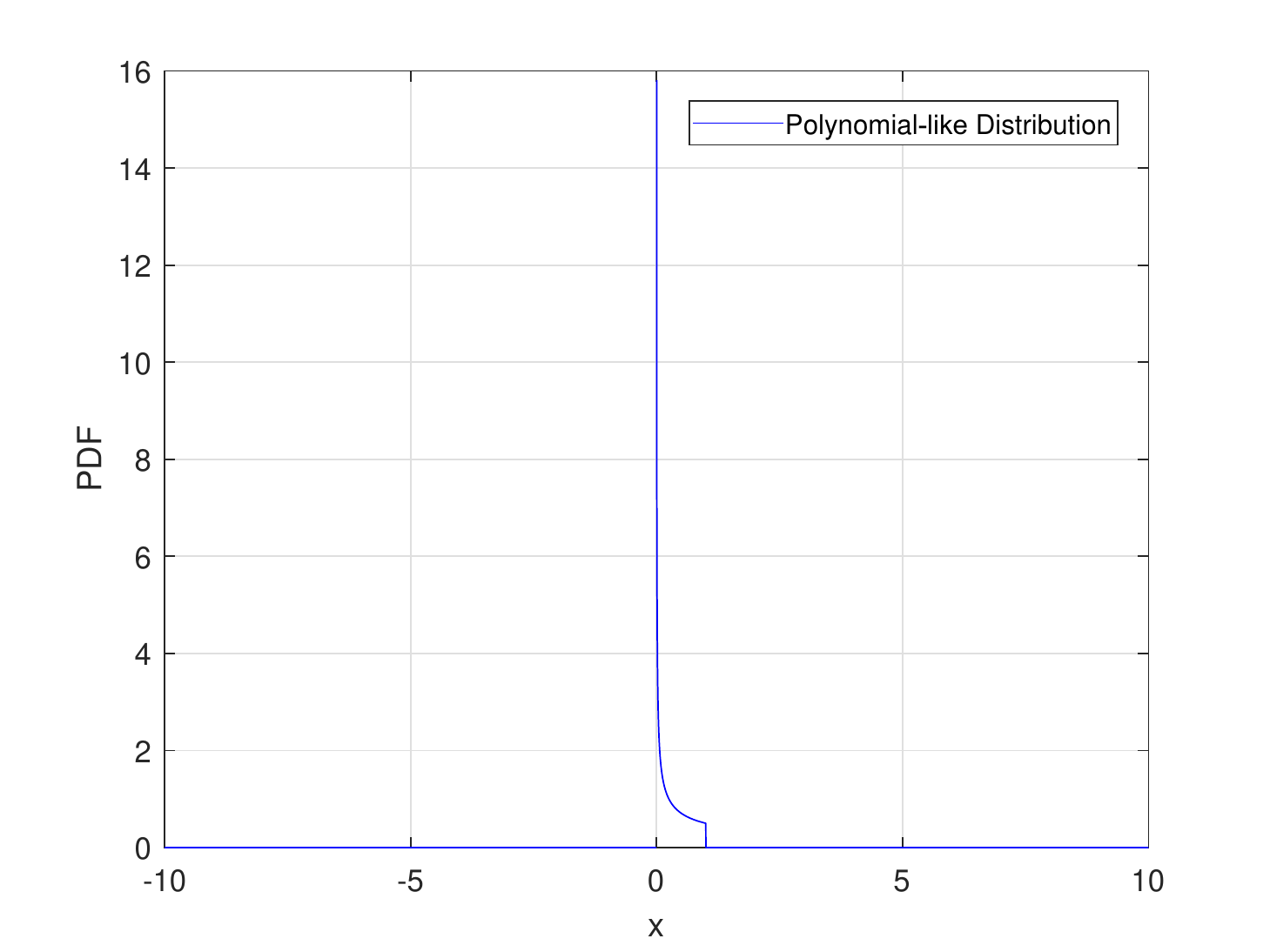}

\includegraphics[width=0.24\columnwidth]{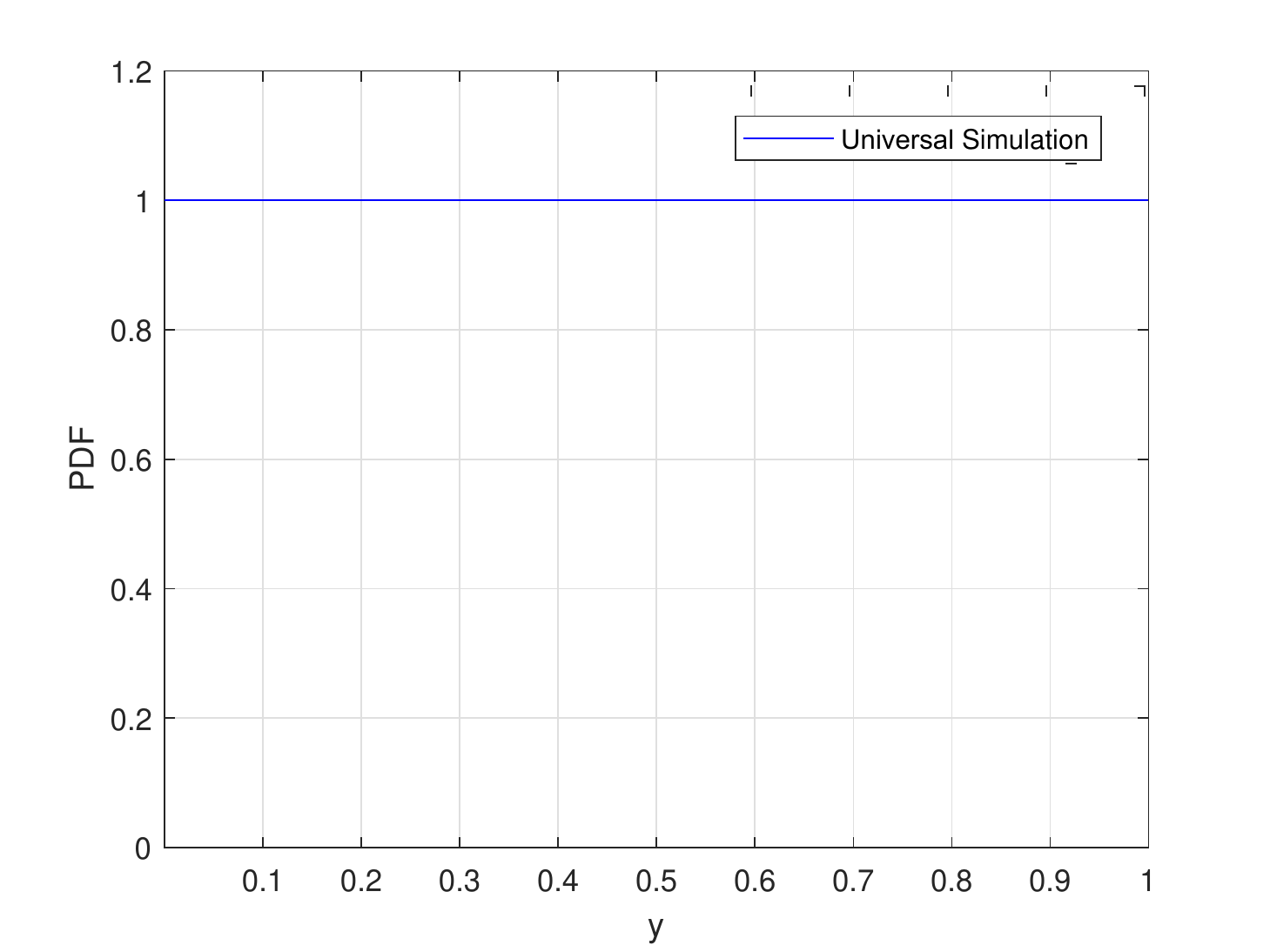} \includegraphics[width=0.24\columnwidth]{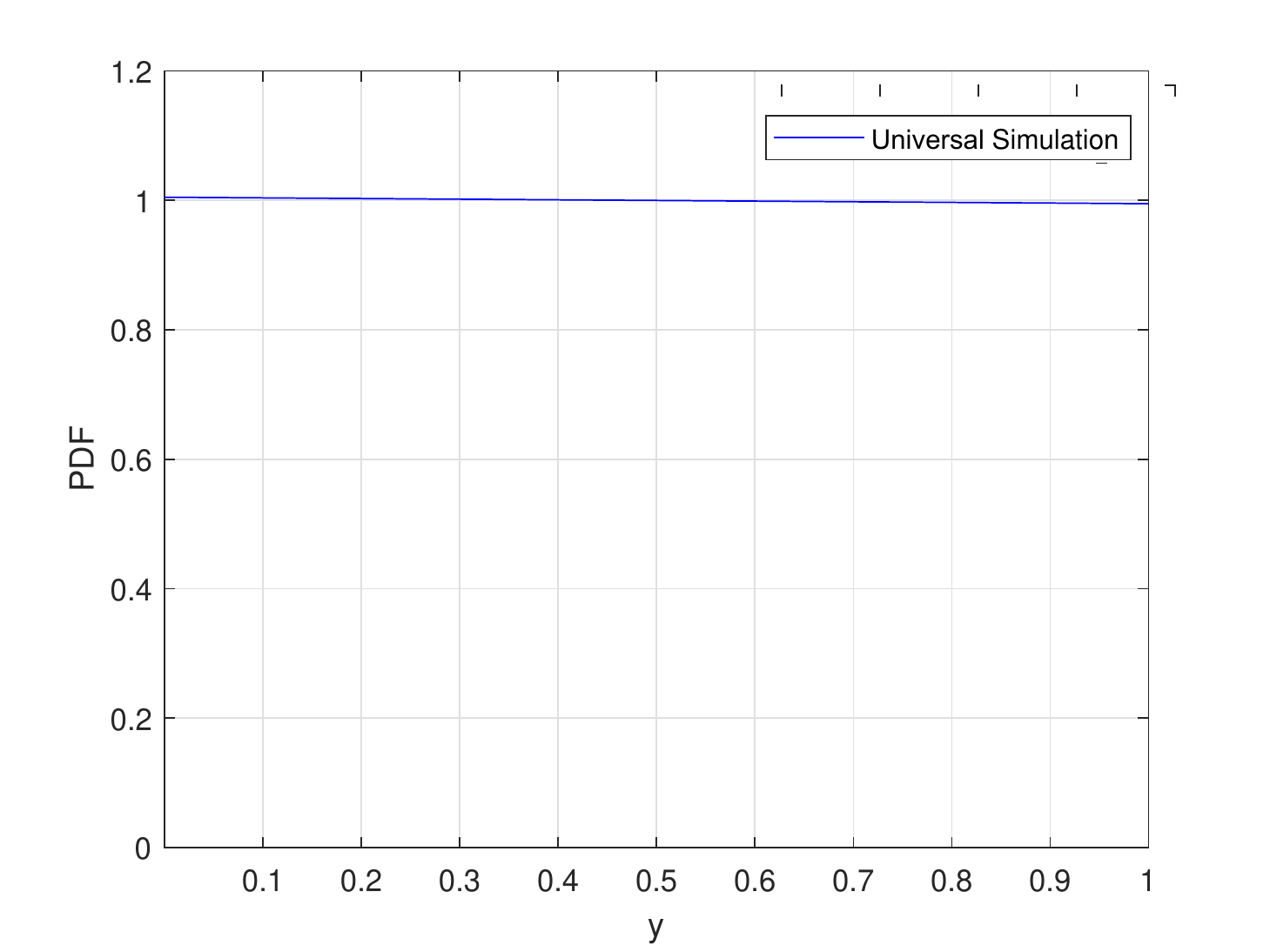}
\includegraphics[width=0.24\columnwidth]{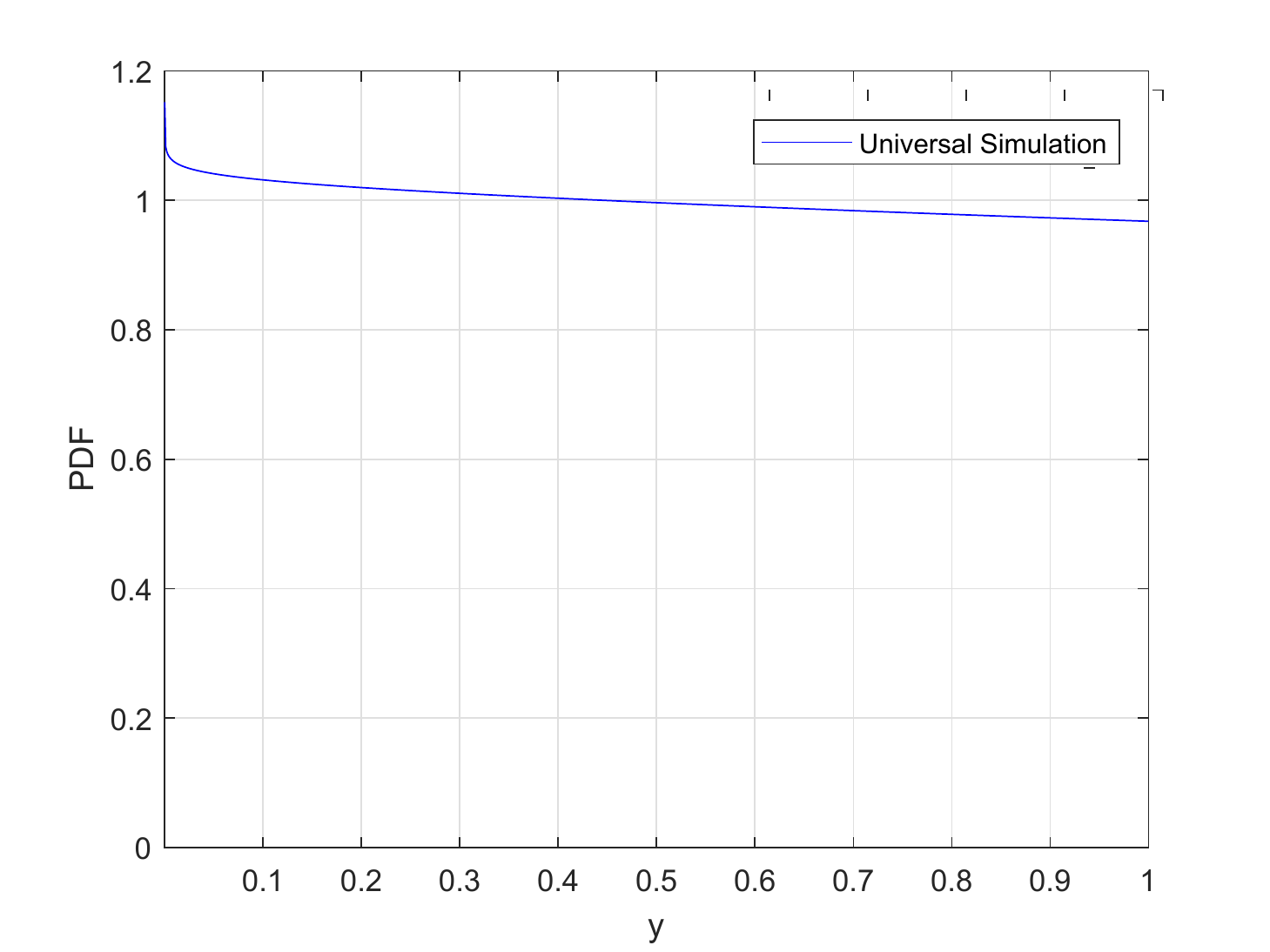} \includegraphics[width=0.24\columnwidth]{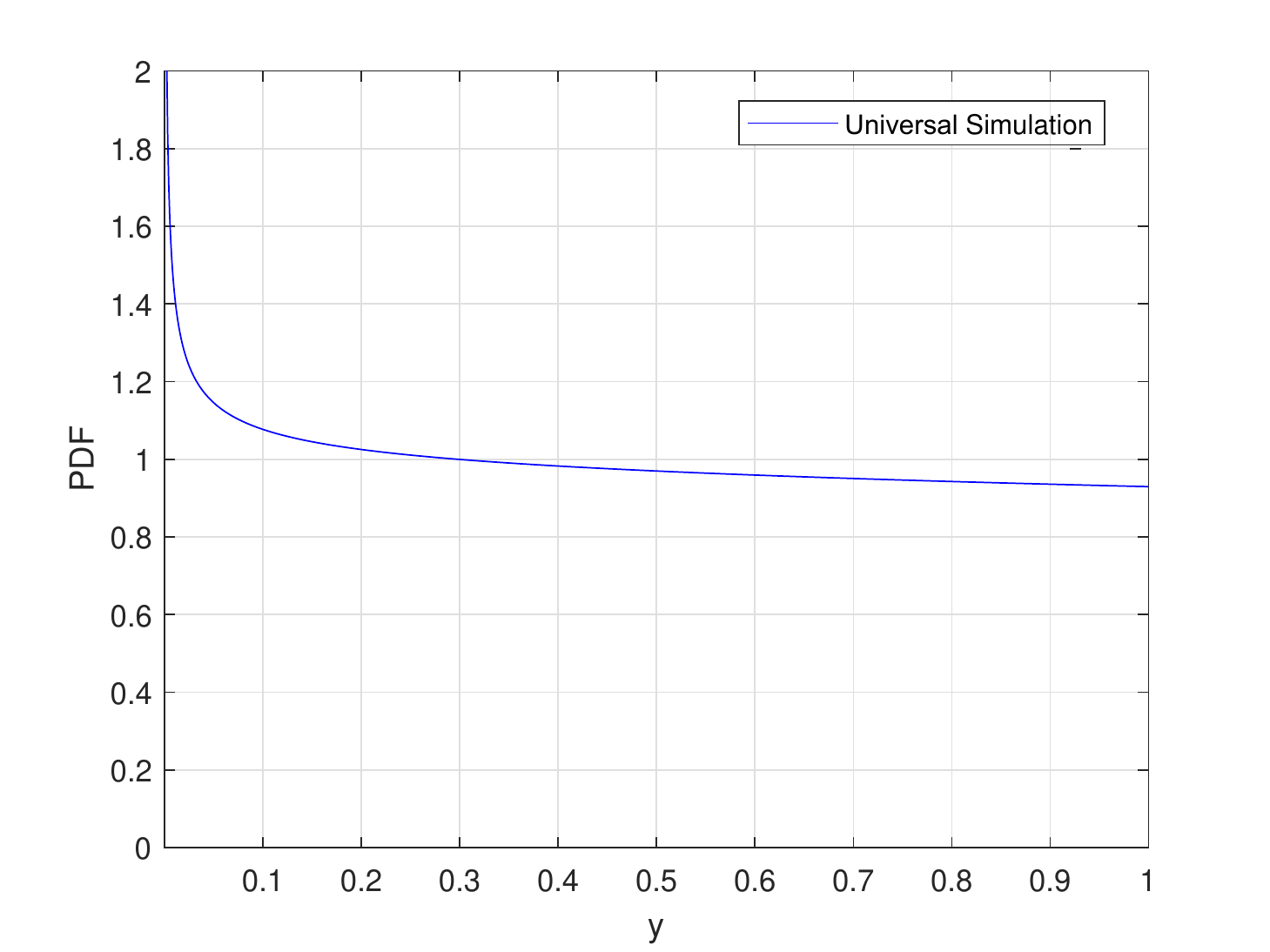}
\caption{Universal simulations in Theorem \ref{thm:Ucontinuous} for the uniform
distribution on $[0,1]$ by the mapping $x\protect\mapsto\sum_{i=-\infty}^{\infty}\frac{1}{\Delta}(x-i\Delta)1\left\{ x\in(i\Delta,(i+1)\Delta]\right\} $
with $\Delta=0.01$. The seed distributions $P_{X}$ are illustrated
in top figures, which are respectively the standard Gaussian distribution
$\mathcal{N}(0,1)$, exponential distribution $\mathrm{Exp}(1)$,
logarithmic distribution $p_{X}(x)=-\log x,x\in(0,1]$, and polynomial-like
distribution $p_{X}(x)=\left(1-r\right)x^{-r},x\in(0,1],r=0.5$. The
generated distributions $P_{Y}$ are illustrated in bottom figures.}
\label{fig:blind} 
\end{figure}

\subsubsection{Relative Entropy and Rényi Divergence Measures}

Next we extend Theorem \ref{thm:Ucontinuous} to the relative entropy
and Rényi divergence measures. The relative entropy and Rényi divergence
are two information measures that quantify the ``distance'' between
probability measures. 

Fix distributions $P_{X},Q_{X}\in\mathcal{P}(\mathbb{R},\mathcal{B}_{\mathbb{R}})$.
The {\em relative entropy} and the {\em Rényi divergence of order 
$\alpha \in (0,1)\cup(1,\infty)$} are respectively defined as
\begin{align*}
D(P_{X}\|Q_{X}) & :=\int\left(\frac{dP_{X}}{dQ_{X}}\log\frac{dP_{X}}{dQ_{X}}\right)dQ_{X}\\*
D_{\alpha}(P_{X}\|Q_{X}) & :=\frac{1}{\alpha-1}\log\int\left(\frac{dP_{X}}{dQ_{X}}\right)^{\alpha}dQ_{X},
\end{align*}
and the conditional versions are respectively defined as 
\begin{align*}
D(P_{Y|X}\|Q_{Y|X}|P_{X}) & :=D(P_{X}P_{Y|X}\|P_{X}Q_{Y|X})\\*
D_{\alpha}(P_{Y|X}\|Q_{Y|X}|P_{X}) & :=D_{\alpha}(P_{X}P_{Y|X}\|P_{X}Q_{Y|X}).
\end{align*}
The  Rényi divergence of order $\alpha \in \{0,1,\infty\}$ is defined by continuous extension.
 Throughout, $\log$ and $\exp$ are to the natural base $e$. It is
known that $D_{1}(P_{X}\|Q_{X}):=\lim_{\alpha\to1}D_{\alpha}(P_{X}\|Q_{X})=D(P_{X}\|Q_{X})$
so a special case of the Rényi divergence (resp. the conditional version)
is the usual relative entropy (resp. the conditional version). The
Rényi divergence of infinity order is defined as 
\begin{align*}
D_{\infty}(P_{X}\|Q_{X}) & :=\lim_{\alpha\to\infty}D_{\alpha}(P_{X}\|Q_{X})=\log\esssup_{P_{X}}\frac{dP_{X}}{dQ_{X}}.
\end{align*}
\begin{defn}
A function $g:\mathcal{P}_{X}\to\mathbb{R}$ is called \emph{$\alpha$-Rényi-achievable}
for the \emph{universal} $(\mathcal{P}_{X},Q_{Y})$-simulation, if
there exists a sequence of simulators $\left\{ f_{k}\right\} _{k=1}^{\infty}$
such that 
\begin{equation}
\limsup_{k\to\infty}D_{\alpha}(P_{Y_{k}}\|Q_{Y})\leq g(P_{X})\label{eq:-2}
\end{equation}
for all $P_{X}\in\mathcal{P}_{X}$, where $P_{Y_{k}}:=P_{X}\circ f_{k}^{-1}$. 
\end{defn}
\begin{defn}
\emph{The set of $\alpha$-Rényi-achievable functions} for the \emph{universal}
$(\mathcal{P}_{X},Q_{Y})$-simulation is defined as 
\[
\mathcal{E}_{\mathrm{Renyi}}^{(\alpha)}(\mathcal{P}_{X},Q_{Y}):=\left\{ g:\mathcal{P}_{X}\to\mathbb{R}\::g\textrm{ is }\alpha\raisebox{.75pt}{-}\textrm{Rényi}\raisebox{.75pt}{-}\textrm{achievable}\right\} .
\]
\end{defn}
For $\left(\mathcal{A},\mathcal{B}_{\mathcal{A}}\right)\subset\left(\mathbb{R},\mathcal{B}_{\mathbb{R}}\right)$,
define 
\[
\mathcal{P}_{\mathrm{ac}}^{(\alpha)}\left(\mathcal{A},\mathcal{B}_{\mathcal{A}}\right):=\left\{ P_{X}\in\mathcal{P}_{\mathrm{ac}}\left(\mathcal{A},\mathcal{B}_{\mathcal{A}}\right):\lim_{\Delta\to0}\sup_{\left\{ \mathcal{A}_{i}\right\} :\sup_{i}\left|\mathcal{A}_{i}\right|\le\Delta}D_{\alpha}(P_{X|\left\{ \mathcal{A}_{i}\right\} }\|P_{X})=0\right\} ,
\]
where the supremum is taken over all partitions $\left\{ \left\{ \mathcal{A}_{i}\right\} :\mathcal{A}_{i}\textrm{ are intervals, and}\  \bigcup_{i}\mathcal{A}_{i}=\mathcal{A},\ \mathcal{A}_{i}\cap\mathcal{A}_{j}=\emptyset,\forall i\neq j\right\} $
of $\mathcal{A}$, $\left|\mathcal{A}_{i}\right|$ is the length of
the interval $\mathcal{A}_{i}$, and $P_{X|\left\{ \mathcal{A}_{i}\right\} }:=\frac{1\left\{ x\in\mathcal{A}_{i}\right\} }{\left|\mathcal{A}_{i}\right|}P_{X}\left(\mathcal{A}_{i}\right)$.
Based on the notations above, we have the following theorem.
\begin{thm}
\label{thm:Ucontinuous-1} Assume $\left(\mathcal{A},\mathcal{B}_{\mathcal{A}}\right)\subset\left(\mathbb{R},\mathcal{B}_{\mathbb{R}}\right)$,
\[
\mathcal{P}_{X}=\begin{cases}
\mathcal{P}_{\mathrm{ac}} & \alpha\in[0,1)\\
\mathcal{P}_{\mathrm{ac}}^{(\alpha)}\left(\mathcal{A},\mathcal{B}_{\mathcal{A}}\right) & \alpha\in[1,\infty]
\end{cases},
\]
and $Q_{Y}$ is arbitrary. Then for the universal $(\mathcal{P}_{X},Q_{Y})$-simulation
problem, $\mathcal{E}_{\mathrm{Renyi}}^{(\alpha)}(\mathcal{P}_{X},Q_{Y})\asymp0$. 
\end{thm}
\begin{rem}
For $0<a\le b<\infty$, 
\begin{align*}
\mathcal{P}_{\mathrm{ac}}^{[a,b]}\left(\mathcal{A},\mathcal{B}_{\mathcal{A}}\right) & :=\left\{ P_{X}\in\mathcal{P}_{\mathrm{ac}}:a\le p(x)\le b\textrm{ a.e. in }\mathcal{A},\textrm{ and }P_{X}(\mathbb{R}\backslash\mathcal{A})=0\right\} \\
 & \subset\mathcal{P}_{\mathrm{ac}}^{(\alpha)}\left(\mathcal{A},\mathcal{B}_{\mathcal{A}}\right),\alpha\in[1,\infty]
\end{align*}
Hence Theorem \ref{thm:Ucontinuous-1} still holds for $\mathcal{P}_{X}=\mathcal{P}_{\mathrm{ac}}^{[a,b]}\left(\mathcal{A},\mathcal{B}_{\mathcal{A}}\right)$
and $\alpha\in[1,\infty]$. Furthermore, Theorem \ref{thm:Ucontinuous-1}
also holds if in \eqref{eq:-2} $D_{\alpha}(P_{Y_{k}}\|Q_{Y})$ is
replaced with $D_{\alpha}(Q_{Y}\|P_{Y_{k}})$, $\alpha\in[0,\infty]$. 
\end{rem}
\begin{proof}
It is easy to verify that \eqref{eq:-1-2} with the total variable
distance replaced by the Rényi divergence still holds.  This implies
the cases $\alpha\in[1,\infty]$ in Theorem \ref{thm:Ucontinuous-1}.
The cases $\alpha\in[0,1)$ are implied by combining Theorem \ref{thm:Ucontinuous}
with the following lemma which shows the ``equivalence'' between
the Rényi divergence of order $\alpha\in(0,1)$ and the TV distance,
in the sense that $D_{\alpha}(P_{X}\|Q_{X})\to0$ if and only if $\left|P_{X}-Q_{X}\right|_{\mathrm{TV}}\to0$. 
\end{proof}
\begin{lem}
For $\alpha\in(0,1)$,
\[
\frac{1}{\alpha-1}\log\left(1+\frac{1}{2}\left|P_{X}-Q_{X}\right|_{\mathrm{TV}}\right)\leq D_{\alpha}(P_{X}\|Q_{X})\leq\frac{1}{\alpha-1}\log\left(1-\frac{1}{2}\left|P_{X}-Q_{X}\right|_{\mathrm{TV}}\right).
\]
\end{lem}
\begin{proof}
This lemma follows from the following two inequalities. Define $A:=\left\{ x:P_{X}(x)\ge Q_{X}(x)\right\} $.
Then 
\begin{align*}
D_{\alpha}(P_{X}\|Q_{X}) & =\frac{1}{\alpha-1}\log\left(\int_{A}\left(\frac{dP_{X}}{dQ_{X}}\right)^{\alpha}\mathrm{d}Q_{X}+\int_{\mathbb{R}\backslash A}\left(\frac{dQ_{X}}{dP_{X}}\right)^{1-\alpha}\mathrm{d}P_{X}\right)\\
 & \leq\frac{1}{\alpha-1}\log\left(\int_{A}dQ_{X}+\int_{\mathbb{R}\backslash A}dP_{X}\right)\\
 & =\frac{1}{\alpha-1}\log\left(1-\left(P_{X}\left(A\right)-Q_{X}\left(A\right)\right)\right)\\
 & =\frac{1}{\alpha-1}\log\left(1-\frac{1}{2}\left|P_{X}-Q_{X}\right|_{\mathrm{TV}}\right).
\end{align*}
Similarly, 
\begin{align*}
D_{\alpha}(P_{X}\|Q_{X}) & \geq\frac{1}{\alpha-1}\log\left(1+\frac{1}{2}\left|P_{X}-Q_{X}\right|_{\mathrm{TV}}\right).
\end{align*}
\end{proof}

\subsection{Discontinuous seed distributions}

Next we consider the case $\mathcal{P}_{X}\subseteq\mathcal{P}_{\mathrm{dc}}$.
We first derive a discontinuous version of Theorem \ref{thm:Ucontinuous}.
Since in the previous subsection, Theorem \ref{thm:Ucontinuous} is
proven only for absolutely continuous seed distributions, one may
doubt the effectiveness of the proposed universal mapping in Fig.
\ref{fig:Illustration-of-the} when the seed distributions are discontinuous
or even discrete. In the following, we prove that our proposed universal
mapping still works well for discontinuous seed distributions, as
long as the CDF of seed distribution is smooth enough.

Assume $\left\{ \Delta_{n}\right\} $ is a sequence of non-increasing
positive numbers. Assume $\left\{ \mathcal{P}_{X_{n}}\right\} _{n=1}^{\infty}$
is a sequence of distribution sets such that for every sequence $\left\{ P_{X_{n}}\right\} \in\left\{ \mathcal{P}_{X_{n}}\right\} $,
its CDFs $\left\{ F_{X_{n}}\right\} $ satisfy
\begin{equation}
\lim_{n\to\infty}\sup_{x_{1}:F_{X_{n}}(x_{1}+\Delta_{n})>F_{X_{n}}(x_{1})}\sup_{x\in(x_{1},x_{1}+\Delta_{n}]}\left|\frac{F_{X_{n}}(x_{1}+x)-F_{X_{n}}(x_{1})}{F_{X_{n}}(x_{1}+\Delta_{n})-F_{X_{n}}(x_{1})}-\frac{x}{\Delta_{n}}\right|=0.\label{eq:-26}
\end{equation}
We obtain a discontinuous version of Theorem \ref{thm:Ucontinuous}.
The proof of Proposition \ref{prop:Udiscontinuous} is provided in
Appendix \ref{sec:Proof-of-Theorem-discrete-1}. 
\begin{prop}
\label{prop:Udiscontinuous} Assume $\left\{ \mathcal{P}_{X_{n}}\right\} $
is the sequence of distribution sets defined above. Then there exists
a sequence of universal mappings $Y_{n}=f_{n}(X_{n})$ (which are
dependent on $\left\{ \Delta_{n}\right\} $) such that $\lim_{n\to\infty}\left|P_{Y_{n}}-Q_{Y}\right|_{\mathrm{KS}}=0,\forall\left\{ P_{X_{n}}\right\} \in\left\{ \mathcal{P}_{X_{n}}\right\} $.
That is, $\lim_{n\to\infty}\mathcal{E}_{\mathrm{KS}}(\mathcal{P}_{X_{n}},Q_{Y})\asymp0$. 
\end{prop}
The proposition above implies that if the CDF of seed random variable
$X_{n}$ gets more and more smooth as $n\to\infty$ in the sense that
it can be approximated by a linear function for every small interval
$(x_{1},x_{1}+\Delta_{n}]$ as \eqref{eq:-26}, then we can find a
sequence of universal mappings that achieve vanishing KS-approximation
error. Here is a simple example. 
\begin{example}
\label{exa:Assume--is}Assume $X$ is an absolutely continuous random
variable with a bounded PDF. We define $X_{n}:=\frac{\left\lfloor nX\right\rfloor }{n}$
as a quantized version of $X$ with quantization step $\frac{1}{n}$,
and $\Delta_{n}$ is set to $\frac{1}{\sqrt{n}}$, then $\left(\left\{ X_{n}\right\} ,\left\{ \Delta_{n}\right\} \right)$
satisfies \eqref{eq:-26}.
\end{example}
The example above with $\frac{1}{n}=0.0001$, $\Delta=0.01$, and
$P_{X}$ to be the standard Gaussian distribution or logarithmic distribution,
is illustrated in Fig. \ref{fig:blind-discrete}.

\begin{figure}[t]
\centering 

\includegraphics[width=0.45\columnwidth]{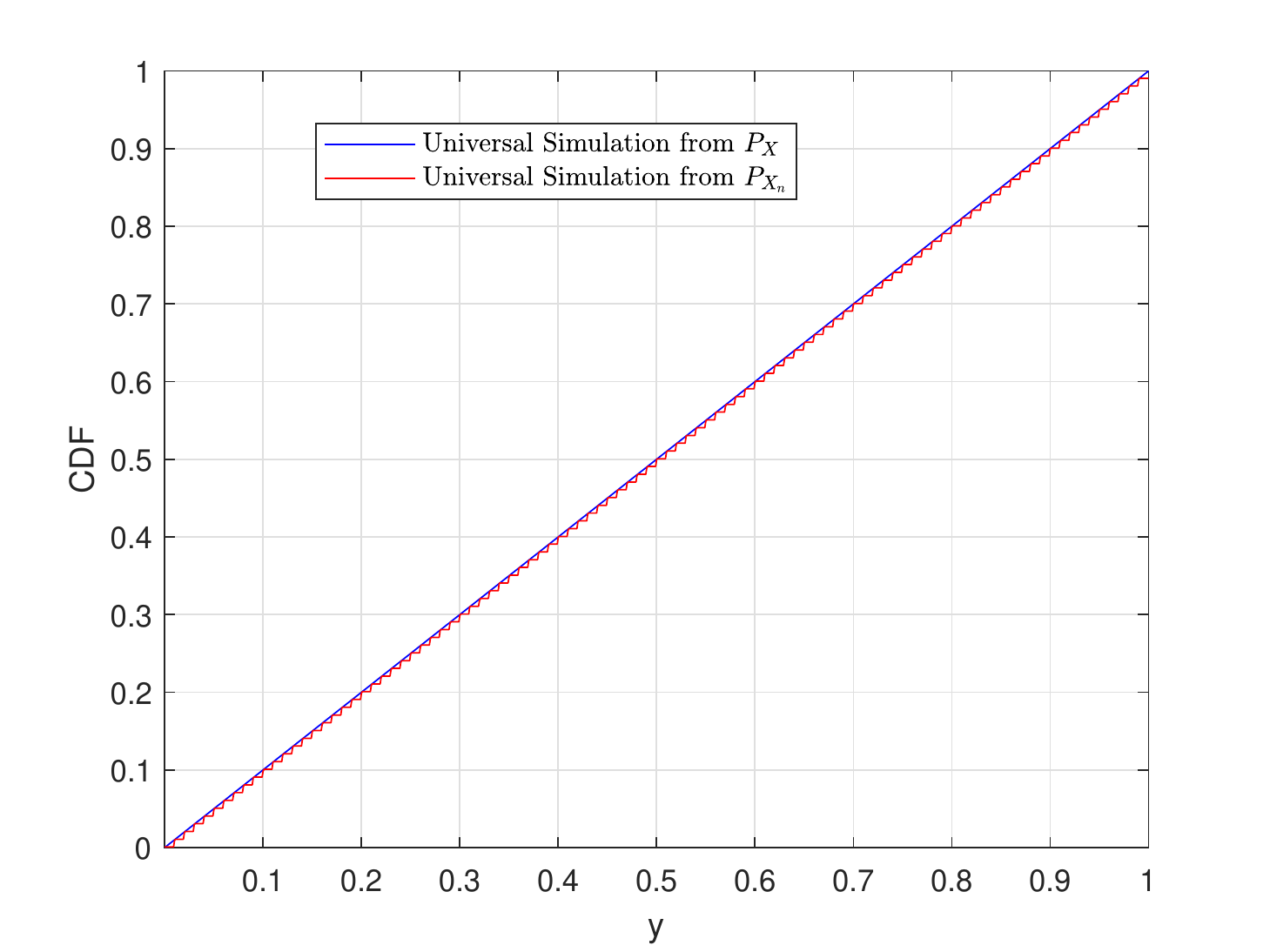} \includegraphics[width=0.45\columnwidth]{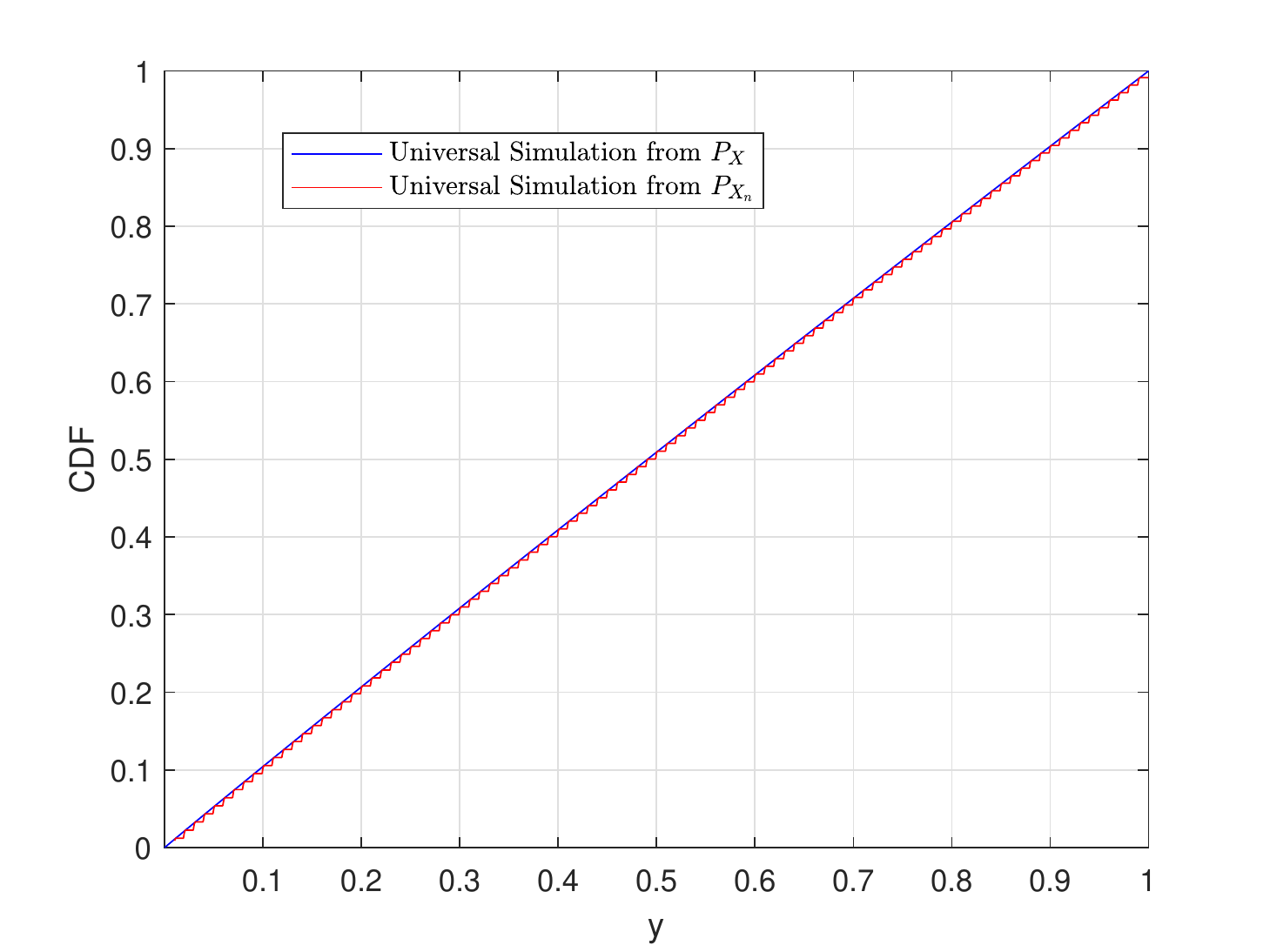} 

\caption{Universal simulations in Proposition \ref{prop:Udiscontinuous} for
the uniform distribution on $[0,1]$ by the mapping $x\protect\mapsto\sum_{i=-\infty}^{\infty}\frac{1}{\Delta}(x-i\Delta)1\left\{ x\in(i\Delta,(i+1)\Delta]\right\} $
with $\Delta=0.01$. The blue and red curves in the left figure correspond
to the cases in which the seed distributions $P_{X}$ are respectively
  the standard Gaussian distribution $\mathcal{N}(0,1)$ and its
quantized version. The blue and red curves in the right figure correspond
to the cases in which the seed distributions $P_{X}$ are respectively
  the logarithmic distribution $p_{X}(x)=-\log x,x\in(0,1]$ and
its quantized version.  For both of these two cases, the quantized
versions are generated by using the same quantization step $\frac{1}{n}=0.0001$.}
\label{fig:blind-discrete} 
\end{figure}

If the seed is a sequence of i.i.d. discrete random vectors, then
the approximation error decays exponentially fast. Given a Borel subset
$\left(\mathcal{X},\mathcal{B}_{\mathcal{X}}\right)\subseteq\left(\mathbb{R},\mathcal{B}_{\mathbb{R}}\right)$
with $\mathcal{X}$ countable, $\mathcal{P}(\mathcal{X},\mathcal{B}_{\mathcal{X}})$
denotes the set of distributions on $\left(\mathcal{X},\mathcal{B}_{\mathcal{X}}\right)$. 
\begin{thm}
\label{thm:Udiscrete} Assume $\mathcal{P}_{X}=\mathcal{P}(\mathcal{X},\mathcal{B}_{\mathcal{X}})$
and $Q_{Y}$ is continuous. Then for the universal $(\mathcal{P}_{X}^{(n)},Q_{Y})$-simulation
problem, $\mathcal{E}_{\mathrm{KS}}(\mathcal{P}_{X}^{(n)},Q_{Y})\dotasymp\left(\max_{x}P_{X}(x)\right)^{n}$. 
\end{thm}
\begin{proof}
We first consider the case in which $\mathcal{X}$ is a finite set.
We use a type-based mapping scheme to prove Theorem \ref{thm:Udiscrete}.
Here we adopt the notation from \cite{Csiszar}. We use $T_{x^{n}}\left(x\right):=\frac{1}{n}\sum_{i=1}^{n}1\left\{ x_{i}=x\right\} $
to denote the type (empirical distribution) of a sequence $x^{n}$,
and $T_{X}$ to denote a type of sequences in $\mathcal{X}^{n}$,
where the indicator function $1\{A\}$ equals $1$ if the clause $A$
is true and $0$ otherwise. For a type $T_{X}$, the type class (set
of sequences having the same type $T_{X}$) is denoted by $\mathcal{T}_{T_{X}}$.
The set of types of sequences in $\mathcal{X}^{n}$ is denoted as
$\mathcal{P}_{n}\left(\mathcal{X}\right):=\left\{ T_{x^{n}}:x^{n}\in\mathcal{X}^{n}\right\} $.
It has been shown that $\left|\mathcal{P}_{n}\left(\mathcal{X}\right)\right|\leq\left(n+1\right)^{\left|\mathcal{X}\right|}$
in \cite{Csiszar}.

For any i.i.d. $X^{n}$, all sequences in a type class have a equal
probability. That is, under the condition $X^{n}\in\mathcal{T}_{T_{X}}$,
it is uniformly distributed over the type class $\mathcal{T}_{T_{X}}$,
regardless of $P_{X}$. Now we construct a mapping $f$ that maps
the uniform random vector on $\mathcal{T}_{T_{X}}$ to a random variable
such that $\sup_{y\in\mathbb{R}}\left|F_{Y}(y|T_{X})-G_{Y}(y)\right|$
is minimized. Here $F_{Y}(y|T_{X})$ denotes the CDF of the output
random variable for the type $T_{X}$. Since the probability values
of uniform random vectors are all equal to $\left|\mathcal{T}_{T_{X}}\right|^{-1}$,
$\sup_{y\in\mathbb{R}}\left|F_{Y}(y|T_{X})-G_{Y}(y)\right|=\frac{1}{2}\left|\mathcal{T}_{T_{X}}\right|^{-1}$.
Therefore, the output distribution induced by $f$ is 
\begin{align}
F_{Y}(y) & =\sum_{T_{X}}\sum_{x^{n}\in\mathcal{T}_{T_{X}}}P_{X}^{n}(x^{n})1\{f(x^{n})\leq y\}\label{eq:-20}\\
 & =\sum_{T_{X}}P_{X}^{n}(\mathcal{T}_{T_{X}})\sum_{x^{n}\in\mathcal{T}_{T_{X}}}\frac{P_{X}^{n}(x^{n})}{P_{X}^{n}(\mathcal{T}_{T_{X}})}1\{f(x^{n})\leq y\}\nonumber \\
 & =\sum_{T_{X}}P_{X}^{n}(\mathcal{T}_{T_{X}})\sum_{x^{n}\in\mathcal{T}_{T_{X}}}\frac{1}{\left|\mathcal{T}_{T_{X}}\right|}1\{f(x^{n})\leq y\}\nonumber \\
 & =\sum_{T_{X}}P_{X}^{n}(\mathcal{T}_{T_{X}})F_{Y}(y|T_{X})\nonumber \\
 & \in\sum_{T_{X}}P_{X}^{n}(\mathcal{T}_{T_{X}})\left(G_{Y}(y)+\left[-\frac{1}{2}\left|\mathcal{T}_{T_{X}}\right|^{-1},\frac{1}{2}\left|\mathcal{T}_{T_{X}}\right|^{-1}\right]\right)\nonumber \\
 & =G_{Y}(y)+\sum_{T_{X}}P_{X}^{n}(\mathcal{T}_{T_{X}})\left[-\frac{1}{2}\left|\mathcal{T}_{T_{X}}\right|^{-1},\frac{1}{2}\left|\mathcal{T}_{T_{X}}\right|^{-1}\right].\nonumber 
\end{align}
Using this equation we obtain 
\begin{align}
\left|F_{Y}(y)-G_{Y}(y)\right| & \leq\frac{1}{2}\sum_{T_{X}}P_{X}^{n}(\mathcal{T}_{T_{X}})\left|\mathcal{T}_{T_{X}}\right|^{-1}\nonumber \\
 & =\frac{1}{2}\sum_{T_{X}}e^{n\sum_{x}T_{X}(x)\log P_{X}(x)}\nonumber \\
 & \leq\frac{1}{2}\left(n+1\right)^{\left|\mathcal{X}\right|}\max_{T_{X}}e^{n\sum_{x}T_{X}(x)\log P_{X}(x)}\nonumber \\
 & \leq\frac{1}{2}e^{n\left(\log\max_{x}P_{X}(x)+\left|\mathcal{X}\right|\frac{\log\left(n+1\right)}{n}\right)}\label{eq:-30}\\
 & \doteq e^{n\log\max_{x}P_{X}(x)}\nonumber \\
 & =\left(\max_{x}P_{X}(x)\right)^{n}\label{eq:-21}
\end{align}

We next consider the case in which $\mathcal{X}$ is countably infinite.
For brevity, we assume $\mathcal{X}=\mathbb{Z}$. We partition $\mathbb{Z}$
into $2k+1$ intervals\footnote{Sometimes, we use $[a:b]$ to denote $\mathbb{Z}\cap[a,b]$. }
$\mathcal{U}_{-k}:=[-\infty:-k]$, $\mathcal{U}_{-\left(k-1\right)}:=\left\{ -\left(k-1\right)\right\} $,
..., $\mathcal{U}_{k-1}:=\left\{ k-1\right\} $, $\mathcal{U}_{k}:=[k:\infty]$.
Denote $Z_{k}=f_{1,k}(X)\in\mathcal{Z}:=[-k:k]$ as the index that
$X\in\mathcal{U}_{Z_{k}}$. Hence $P_{Z_{k}}$ is defined on the finite
set $\mathcal{Z}$. Now we use $Z_{k}$ to simulate $Y\sim Q_{Y}$.
By the derivation above, we have that there exists a universal mapping
$Y_{k}=f_{2,k}(Z_{k}):\mathcal{Z}\to\mathcal{Y}$ such that $\left|P_{Y_{k}}-Q_{Y}\right|_{\mathrm{KS}}\doteq\left(\max_{z}P_{Z_{k}}(z)\right)^{n}$.
Furthermore, as $k\to\infty$, $\max_{z}P_{Z_{k}}(z)\to\max_{x}P_{X}(x)$.
Therefore, the universal mappings $f_{2,k}\circ f_{1,k},k\in\mathbb{Z}$
satisfy $\limsup_{k\to\infty}\left|P_{Y_{k}}-Q_{Y}\right|_{\mathrm{KS}}\dotle\left(\max_{x}P_{X}(x)\right)^{n}$.

The converse part follows from Theorem \ref{prop:NUdiscrete}, since
even non-universal simulation cannot make the approximation error
decay faster than $\left(\max_{x}P_{X}(x)\right)^{n}$, hence universal
simulation cannot as well. 
\end{proof}
Now we consider a discontinuous $P_{X}$. We partition the real line
into intervals $\mathcal{U}_{k}:=((k-1)\Delta,k\Delta],k\in\mathbb{Z}$
with $\Delta=\frac{1}{\sqrt{k}}$. Denote $Z_{k}=f_{1,k}(X)\in\mathbb{Z}$
as the index that $X\in\mathcal{U}_{Z_{k}}$. Hence $P_{Z_{k}}$ is
defined on the set $\mathbb{Z}$. Now we use $Z_{k}$ to simulate
$Y\sim Q_{Y}$. By Theorem \ref{thm:Udiscrete}, we have that there
exists a universal mapping $Y_{k}=f_{2,k}(Z_{k}):\mathcal{Z}\to\mathcal{Y}$
such that $\left|P_{Y_{k}}-Q_{Y}\right|_{\mathrm{KS}}\doteq\left(\max_{z}P_{Z_{k}}(z)\right)^{n}$.
Furthermore, as $k\to\infty$, we have $\max_{z}P_{Z_{k}}(z)\to\max_{x}P_{X}(x)$.
Therefore, the universal mappings $f_{2,k}\circ f_{1,k},k\in\mathbb{Z}$
satisfy $\limsup_{k\to\infty}\left|P_{Y_{k}}-Q_{Y}\right|_{\mathrm{KS}}\dotle\left(\max_{x}P_{X}(x)\right)^{n}$.
Therefore, we have the following result. 
\begin{cor}
\label{cor:Udiscontinuous} Assume $\mathcal{P}_{X}=\mathcal{P}_{\mathrm{dc}}$
and $Q_{Y}$ is continuous. Then for the universal $(\mathcal{P}_{X}^{(n)},Q_{Y})$-simulation
problem, $\mathcal{E}_{\mathrm{KS}}(\mathcal{P}_{X}^{(n)},Q_{Y})\dotasymp\left(\max_{x}P_{X}(x)\right)^{n}$. 
\end{cor}

\subsection{Continuous but not absolutely continuous seed distributions}

Next we consider continuous but not absolutely continuous $P_{X}$.
For this case, we have $\max_{x}P_{X}(x)=0$. Hence the approximation
error decays sup-exponentially fast. To provide a better bound, we
assume $F_{X}$ is Hölder continuous with exponent $\alpha$, where
$0<\alpha\leq1$. That is, $L=\sup_{x_{1}\neq x_{2}}\frac{F_{X}(x_{2})-F_{X}(x_{1})}{\left(x_{2}-x_{1}\right)^{\alpha}}$
is finite. Consider the following mapping. We partition the real line
into $2k+2$ intervals $U_{-k}:=(-\infty,-k\Delta]$, $U_{-\left(k-1\right)}:=(-k\Delta,-\left(k-1\right)\Delta]$,
..., $U_{k}:=((k-1)\Delta,k\Delta]$, $U_{k+1}:=(k\Delta,\infty)$.
Denote $Z=f_{1}(X)\in\mathcal{Z}:=[-k:k+1]$ as the index that $X\in U_{Z}$.
Now we use $Z$ to simulate $Y\sim Q_{Y}$. By the derivation till
\eqref{eq:-30}, we have that there exists a universal mapping $Y=f_{2}(Z):\mathcal{Z}\to\mathcal{Y}$
such that $\left|P_{Y}-Q_{Y}\right|_{\mathrm{KS}}\leq\frac{1}{2}e^{n\left(\log\max_{z}P_{Z}(z)+\left(2k+2\right)\frac{\log\left(n+1\right)}{n}\right)}$.
Set $\Delta=\frac{\log n}{n},k=\frac{n}{\sqrt{\log n}}$. Then $\Delta\to0$
and $k\Delta=\sqrt{\log n}\to\infty$. Since $F_{X}$ is Hölder continuous
with exponent $\alpha$, we have $\max_{z}P_{Z}(z)\leq\max_{x}\left\{ F_{X}(x+\Delta)-F_{X}(x)\right\} \leq L\Delta^{\alpha}$.
Hence the universal mapping $Y_{n}=f_{2}\circ f_{1}(X^{n})$ satisfies
$\left|P_{Y_{n}}-Q_{Y}\right|_{\mathrm{KS}}=e^{-\alpha\Omega\left(n\log n\right)}$.
Therefore, we have the following result. 
\begin{cor}
\label{cor:Uc-ac} Assume $\mathcal{P}_{X}=\mathcal{P}_{\mathrm{c}}\backslash\mathcal{P}_{\mathrm{ac}}$
and $Q_{Y}$ is arbitrary. Then for the universal $(\mathcal{P}_{X}^{(n)},Q_{Y})$-simulation
problem, $\mathcal{E}_{\mathrm{KS}}(\mathcal{P}_{X}^{(n)},Q_{Y})\asymp e^{-\omega\left(n\right)}$.
That is, there exists a sequence of simulators such that $\left|P_{Y^{(n)}}-Q_{Y}\right|_{\mathrm{KS}}$
decays sup-exponentially fast as $n\to\infty$ for any $P_{X}$. Moreover,
if $\mathcal{P}_{X}=\left\{ P_{X}:\:F_{X}\textrm{ is Hölder continuous with exponent }\alpha\right\} $
with $0<\alpha\leq1$, then $\mathcal{E}_{\mathrm{KS}}(\mathcal{P}_{X}^{(n)},Q_{Y})\asymp e^{-\alpha\Omega\left(n\log n\right)}$. 
\end{cor}

\section{Simulating a Random Variable from a Markov Process}

In the preceding sections, we consider simulation of a random variable
from a stationary memoryless process. Next we extend Theorem \ref{thm:Udiscrete}
to Markov processes of order $k\geq1$.
\begin{defn}
Given a Markov chain $\boldsymbol{X}=\left\{ X_{n}:n\in\mathbb{N}\right\} $
of order $k\geq1$ with finite state space $\mathcal{X}=\left\{ 1,2,...,|\mathcal{X}|\right\} $,
initial state $x_{-k+1}^{0}:=(x_{-k+1},x_{-k+2},...,x_{0})$, and
transition probability $P_{X_{k+1}|X^{k}}$, the min-entropy ($\infty$-order
Rényi entropy) rate of $\boldsymbol{X}$ is defined as\footnote{The existence of the limit is guaranteed by Fekete's subadditive lemma.}
\[
H_{\text{\ensuremath{\infty}}}(\boldsymbol{X})=-\lim_{n\to\infty}\frac{1}{n}\max_{x^{n}}\log P(x^{n}).
\]
\end{defn}
Since the distribution of $X^{n}$ is determined by the initial state
$x_{-k+1}^{0}$ and transition probability $P_{X_{k+1}|X^{k}}$, hence
sometimes we also use $H_{\text{\ensuremath{\infty}}}(x_{-k+1}^{0},P_{X_{k+1}|X^{k}})$
to denote $H_{\text{\ensuremath{\infty}}}(\boldsymbol{X})$.

Given a state space $\mathcal{X}$ of any Markov chain $\left\{ X_{n}:n\in\mathbb{N}\right\} $
of order $k=1$ with transition probability $P_{X_{k+1}|X^{k}}$,
a loop is a sequence of distinct states of the chain $(i_{1},i_{2},...,i_{l})$
with $l\ge1$ such that $P_{i_{s},i_{s+1}}>0$ for $s=1,2,...,l$
where $i_{l+1}=i_{1}$. (If $P_{i,i}>0$, then $(i)$ is a loop.)
The set of all loops of length $l$ is denoted by $\mathcal{C}_{l}(P)$.

Let $P$ be the transition matrix of an ergodic Markov chain of order
$k=1$ on a \emph{finite} alphabet $\mathcal{X}$. The min-entropy
rate of this Markov chain is given by \cite{kamath2016estimation}
\[
H_{\text{\ensuremath{\infty}}}(\boldsymbol{X})=\min_{1\leq l\le|\mathcal{X}|}\min\frac{1}{l}\sum_{s=1}^{l}\imath_{X_{2}|X_{1}}(i_{s+1}|i_{s})=\min_{1\leq l\le|\mathcal{X}|}\min\frac{1}{l}\sum_{s=1}^{l}\log\frac{1}{p_{i_{s+1}i_{s}}},
\]
where the inner minimum is taken over all loops $(i_{1},i_{2},...,i_{l})\in\mathcal{C}_{l}(P)$,
and $\imath_{X_{2}|X_{1}}(j|i):=\log\frac{1}{P_{ij}}$.

\subsection{Non-universal Simulation from a Markov Process}

As a direct consequence of Proposition \ref{prop:NUcontinuous}, we
can obtain the approximation error for non-universal simulation from
a Markov process. 
\begin{cor}
\label{cor:NUMarkov}Assume $\boldsymbol{X}=\left\{ X_{n}:n\in\mathbb{N}\right\} $
is a Markov chain of order $k$ with finite state space $\mathcal{X}$,
initial state $x_{-k+1}^{0}$, and transition probability $P_{X_{k+1}|X^{k}}$,
and $Q_{Y}$ is a continuous distribution. Then for the non-universal
$(P_{X^{n}},Q_{Y})$-simulation problem, $E_{\mathrm{KS}}(P_{X^{n}},Q_{Y})\doteq e^{-nH_{\text{\ensuremath{\infty}}}(x_{-k+1}^{0},P_{X_{k+1}|X^{k}})}$. 
\end{cor}

\subsection{Universal Simulation from a Markov Process}

Given a finite state space $\mathcal{X}$, a order $k\geq1$, and
a initial state $x_{-k+1}^{0}$, we denote the set of all possible
distributions of Markov chains with these parameters as $\mathcal{P}_{X_{k+1}|X^{k}}^{(n)}:=\left\{ \prod_{i=1}^{n}P_{X_{k+1}|X^{k}}(x_{i}|x_{i-1},...,x_{i-k}):P_{X_{k+1}|X^{k}}\in\mathcal{P}_{X_{k+1}|X^{k}}\right\} $,
where $\mathcal{P}_{X_{k+1}|X^{k}}$ denotes the set of all possible
transition probability $P_{X_{k+1}|X^{k}}$ (from $\mathcal{X}^{k}$
to $\mathcal{X}$).

We next consider universal simulation from a Markov process. We generalize
Theorem \ref{thm:Udiscrete} to this case. The proof of Theorem \ref{thm:UMarkov}
is provided in Appendix \ref{sec:Proof-of-Theorem-Markov}. 
\begin{thm}
\label{thm:UMarkov} Assume $Q_{Y}$ is a continuous distribution.
Then given a finite state space $\mathcal{X}$, an order $k$, and
an initial state $x_{-k+1}^{0}$, for the universal $(\mathcal{P}_{X_{k+1}|X^{k}}^{(n)},Q_{Y})$-simulation
problem, we have\footnote{Here the definition of $\dotasymp$ is analogous to that for stationary
memoryless processes. } $\mathcal{E}_{\mathrm{KS}}(\mathcal{P}_{X_{k+1}|X^{k}}^{(n)},Q_{Y})\dotasymp e^{-nH_{\text{\ensuremath{\infty}}}(x_{-k+1}^{0},P_{X_{k+1}|X^{k}})}$. 
\end{thm}

\section{Simulating a Random Element from another Random Element}

\subsection{Non-universal Simulation}

Next we show that an arbitrary continuous \emph{random element }(or
\emph{general random variable}) is sufficient to simulate another
arbitrary \emph{random element}. Here random elements is a generalization
of random variable, which may be defined on any non-empty Borel set
in a separable metric space. 
\begin{thm}
\label{thm:NUgeneralRV} Assume $P_{X}$ and $Q_{Y}$ are two distributions
respectively defined on any non-empty Borel sets $(\mathcal{X},\mathcal{B}_{\mathcal{X}})$
and $(\mathcal{Y},\mathcal{B}_{\mathcal{Y}})$ in two Polish spaces
(complete separable metric spaces). If $P_{X}$ is continuous, then
there exists a measurable mapping $Y=f(X)$ such that $P_{Y}=Q_{Y}$.
That is, $E_{\mathrm{TV}}(P_{X},Q_{Y})=0$. 
\end{thm}
\begin{proof}
For any two  Borel subsets of Polish spaces, they are Borel-isomorphic
if and only if they have the same cardinality, which moreover is either
finite, countable, or $\mathfrak{c}$ (the cardinal of the continuum,
that is, of $[0,1]$) \cite{dudley2002real}. Hence for any measurable
space $(\mathcal{X},\mathcal{B}_{\mathcal{X}})$, we can always find
a Borel subset $(\mathcal{W},\mathcal{B}_{\mathcal{W}})$ of $([0,1],\mathcal{B}_{[0,1]})$
such that $(\mathcal{X},\mathcal{B}_{\mathcal{X}})$ and $(\mathcal{W},\mathcal{B}_{\mathcal{W}})$
are Borel-isomorphic. Suppose $\varphi$ is a Borel isomorphism from
$(\mathcal{X},\mathcal{B}_{\mathcal{X}})$ to $(\mathcal{W},\mathcal{B}_{\mathcal{W}})$.
Denote $P_{W}:=P_{X}\circ\varphi^{-1}$. Since $P_{X}$ is continuous
(or atomless), $P_{W}$ must be continuous as well. This is because
if $P_{X}(\varphi^{-1}(w))=P_{W}(w)>0$ for some $w\in[0,1]$, then
it contradicts with the continuity of $P_{X}$. Hence $P_{W}$ is
continuous. (Furthermore, the existence of the Borel isomorphism $\varphi$
can be shown by \cite[Theorem 9.2.2]{bogachev2007measure} as well.)

Similarly, for any measurable space $(\mathcal{Y},\mathcal{B}_{\mathcal{Y}})$,
we can always find a Borel subset $(\mathcal{Z},\mathcal{B}_{\mathcal{Z}})$
of $(\mathbb{R},\mathcal{B}_{\mathbb{R}})$ such that $(\mathcal{Y},\mathcal{B}_{\mathcal{Y}})$
and $(\mathcal{Z},\mathcal{B}_{\mathcal{Z}})$ are Borel-isomorphic.
Suppose $\phi$ is a Borel isomorphism from $(\mathcal{Y},\mathcal{B}_{\mathcal{Y}})$
to $(\mathcal{Z},\mathcal{B}_{\mathcal{Z}})$. Denote $Q_{Z}$ as
the distribution of $\tilde{Z}:=\phi(\tilde{Y})$ with $\tilde{Y}\sim Q_{Y}$.

By Proposition \ref{prop:NUcontinuous}, we know that there exists
a measurable mapping $\eta$ such that $Z:=\eta(W)\sim Q_{Z}$ with
$W\sim P_{W}$. Now consider the mapping $Y=\phi^{-1}\circ\eta\circ\varphi(X)$.
Obviously, $Y=\phi^{-1}(Z)\sim Q_{Y}$.
\end{proof}
Note that random vectors defined on $(\mathbb{R}^{n},\mathcal{B}_{\mathbb{R}^{n}}),n\in\mathbb{Z}^{+}$
are special cases of such random elements. Hence we have the following
corollary. 
\begin{cor}
\label{cor:NUvector}For a continuous $P_{X}$ defined on $(\mathbb{R},\mathcal{B}_{\mathbb{R}})$
and an arbitrary $Q_{Y^{n}}$ defined on $(\mathbb{R}^{n},\mathcal{B}_{\mathbb{R}^{n}}),n\in\mathbb{Z}^{+}$,
there exists a measurable mapping $Y^{n}=f(X):\mathbb{R}\to\mathbb{R}^{n}$
such that $P_{Y^{n}}=Q_{Y^{n}}$. 
\end{cor}

\subsection{Universal Simulation}

Now we generalize Theorem \ref{thm:Ucontinuous} to simulating random
elements. 
\begin{thm}
\label{thm:UgeneralRV} Assume $R_{X}$ and $Q_{Y}$ are two distributions
respectively defined on any non-empty Borel sets $(\mathcal{X},\mathcal{B}_{\mathcal{X}})$
and $(\mathcal{Y},\mathcal{B}_{\mathcal{Y}})$ in two Polish spaces,
and moreover, $R_{X}$ is continuous. $\mathcal{P}_{X}(R_{X})$ denotes
the set of all absolutely continuous distributions (defined on $(\mathcal{X},\mathcal{B}_{\mathcal{X}})$)
respect to $R_{X}$. Then for the universal $(\mathcal{P}_{X}(R_{X}),Q_{Y})$-simulation
problem, $\mathcal{E}_{\mathrm{TV}}(\mathcal{P}_{X}(R_{X}),Q_{Y})\asymp0$. 
\end{thm}
\begin{proof}
Since $R_{X}$ is continuous, as shown in the proof of Theorem \ref{thm:NUgeneralRV},
there exists a Borel isomorphism $\varphi$ from $(\mathcal{X},\mathcal{B}_{\mathcal{X}})$
to a Borel subset $(\mathcal{W},\mathcal{B}_{\mathcal{W}})$ of $([0,1],\mathcal{B}_{[0,1]})$
such that the output distribution $R_{X}\circ\varphi^{-1}$ is continuous.
Denote $R_{Z}$ as the uniform distribution (which is also the Lebesgue
measure) on $([0,1],\mathcal{B}_{[0,1]})$. Then by Proposition \ref{prop:NUcontinuous},
we know that there exists a measurable mapping $\eta:(\mathcal{W},\mathcal{B}_{\mathcal{W}})\to([0,1],\mathcal{B}_{[0,1]})$
such that 
\begin{equation}
R_{X}\circ\varphi^{-1}\circ\eta^{-1}=R_{Z}.\label{eq:-31}
\end{equation}
Hence a random element $\tilde{X}\sim R_{X}$ is mapped to a uniform
random variable $\tilde{Z}\sim R_{Z}$ through the mapping $\tilde{Z}=\eta\circ\varphi(\tilde{X})$.
We define $P_{Z}:=P_{X}\circ\varphi^{-1}\circ\eta^{-1}$, which denotes
the distribution of $Z=\eta\circ\varphi(X)$ where $X\sim P_{X}$.
Since $P_{X}$ is absolutely continuous respect to $R_{X}$, we have
that $P_{Z}$ is absolutely continuous respect to $R_{Z}$ (or the
Lebesgue measure). This is because on one hand, by \eqref{eq:-31},
we have $R_{X}(\varphi^{-1}\circ\eta^{-1}(A))=0$ for any $A$ such
that $R_{Z}(A)=0$; on the other hand, $P_{X}$ is absolutely continuous
respect to $R_{X}$, hence $P_{X}(\varphi^{-1}\circ\eta^{-1}(A))=0$,
i.e., $P_{Z}(A)=0$.

Since $P_{Z}$ is absolutely continuous, by Theorem \ref{thm:Ucontinuous},
we know that there exists a sequence of universal mappings $\tau_{k}:([0,1],\mathcal{B}_{[0,1]})\to([0,1],\mathcal{B}_{[0,1]})$
(independent of $P_{Z}$) such that the resulting approximation error
\[
\lim_{k\to\infty}\left|P_{Z}\circ\tau_{k}^{-1}-R_{Z}\right|_{\mathrm{TV}}=0.
\]
 Observe that $P_{Z}=P_{X}\circ\varphi^{-1}\circ\eta^{-1}$ and $\eta,\varphi$
(only depend on $R_{X}$) are independent of $P_{X}$. Hence the universal
mappings $\tau_{k}\circ\eta\circ\varphi$ satisfy 
\begin{equation}
\lim_{k\to\infty}\left|P_{X}\circ\varphi^{-1}\circ\eta^{-1}\circ\tau_{k}^{-1}-R_{Z}\right|_{\mathrm{TV}}=0.\label{eq:-32}
\end{equation}

Since $R_{Z}$ is continuous, by Theorem \ref{thm:NUgeneralRV}, we
know that then there exists a measurable mapping $\kappa:([0,1],\mathcal{B}_{[0,1]})\to(\mathcal{Y},\mathcal{B}_{\mathcal{Y}})$
such that $R_{Z}\circ\kappa^{-1}=Q_{Y}$.

Now we consider the universal mappings $Y=\kappa\circ\tau_{k}\circ\eta\circ\varphi(X)$.
We have 
\begin{align*}
\left|P_{X}\circ\varphi^{-1}\circ\eta^{-1}\circ\tau_{k}^{-1}\circ\kappa^{-1}-Q_{Y}\right|_{\mathrm{TV}} & =\sup_{A\in\mathcal{B}_{\mathcal{Y}}}\left|P_{X}\circ\varphi^{-1}\circ\eta^{-1}\circ\tau_{k}^{-1}\circ\kappa^{-1}(A)-Q_{Y}(A)\right|\\
 & =\sup_{A\in\mathcal{B}_{\mathcal{Y}}}\left|P_{X}\circ\varphi^{-1}\circ\eta^{-1}\circ\tau_{k}^{-1}(\kappa^{-1}(A))-R_{Z}(\kappa^{-1}(A))\right|\\
 & \leq\sup_{B\in\mathcal{B}_{[0,1]}}\left|P_{X}\circ\varphi^{-1}\circ\eta^{-1}\circ\tau_{k}^{-1}(B)-R_{Z}(B)\right|\\
 & =\left|P_{X}\circ\varphi^{-1}\circ\eta^{-1}\circ\tau_{k}^{-1}-R_{Z}\right|_{\mathrm{TV}}.
\end{align*}
Combining this with \eqref{eq:-32} yields
\[
\lim_{k\to\infty}\left|P_{X}\circ\varphi^{-1}\circ\eta^{-1}\circ\tau_{k}^{-1}\circ\kappa^{-1}-Q_{Y}\right|_{\mathrm{TV}}=0.
\]
\end{proof}
Note that a random vector defined on $(\mathbb{R}^{n},\mathcal{B}_{\mathbb{R}^{n}}),n\in\mathbb{Z}^{+}$
is a special case of such a random element. Furthermore, for any absolutely
continuous (respect to the Lebesgue measure) $P_{X^{n}}$ defined
on $(\mathbb{R}^{n},\mathcal{B}_{\mathbb{R}^{n}}),n\in\mathbb{Z}^{+}$,
it must be absolutely continuous respect to the $n$-dimensional standard
Gaussian distribution (since its PDF is positive for every point in
$\mathbb{R}^{n}$). Hence we have the following corollary.
\begin{cor}
\label{cor:Uvector}For the set $\mathcal{P}_{\mathrm{ac}}$ of absolutely
continuous distributions on $(\mathbb{R},\mathcal{B}_{\mathbb{R}})$
and an arbitrary $Q_{Y^{n}}$ on $(\mathbb{R}^{n},\mathcal{B}_{\mathbb{R}^{n}}),n\in\mathbb{Z}^{+}$,
the approximation errors for the universal $(\mathcal{P}_{\mathrm{ac}},Q_{Y^{n}})$-simulation
problem satisfies $\mathcal{E}_{\theta}(\mathcal{P}_{\mathrm{ac}},Q_{Y^{n}})\asymp0$
for $\theta\in\left\{ \mathrm{KS},\mathrm{TV}\right\} $. 
\end{cor}

\section{Concluding Remarks}

In this paper, motivated by the CLT and other universal simulation
problems in the literature, we consider both universal and non-universal
simulations of random variables with an arbitrary target distribution
$Q_{Y}$ by general mappings. We investigate the fastest convergence
rate of the approximation error for such a problem. One of our interesting
results is that under universal simulation, an \emph{absolutely continuous}
random element (or a general random variable, including random vectors)
respect to some continuous distribution is sufficient to simulate
another random element arbitrarily well. This requirement is a little
stronger than that for non-universal simulation, since under non-universal
simulation, a\emph{ continuous} random element is sufficient to exactly
simulate another random element. Another interesting result is that
when we use a stationary memoryless process or a Markov process to
simulate a random variable by a universal mapping, the approximation
error decays at least exponentially fast with rate $H_{\infty}(P_{X}):=-\log\max_{x}P_{X}(x)$
as the dimension $n$ of $X^{n}$ goes to infinity. Furthermore, as
a byproduct, we also obtain a property on uncorrelation between a
squeezed periodic function and any other integrable function. We think
this topic is of independent interest, and expect it to be further
applied in other problems in the future.

As for application aspects of our results, although practical digital
computers have finite precision for processing or storing datum, as
indicated by Proposition \ref{prop:Udiscontinuous}, our proposed
universal mapping in Fig. \ref{fig:Illustration-of-the}  still works
well on  such digital computers, as long as they have sufficiently
high precision; see the illustration in Fig. \ref{fig:blind-discrete}.

\appendix{}

\section{\label{sec:Proof-of-Theorem-discrete}Proof of Proposition \ref{prop:NUdiscrete}}

Sort the sequences in $\mathcal{X}{}^{n}$ as $x_{1}^{n},x_{2}^{n},...,x_{|\mathcal{X}|^{n}}^{n}$
such that $P_{X}^{n}(x_{1}^{n})\geq P_{X}^{n}(x_{2}^{n})\geq...\geq P_{X}^{n}(x_{|\mathcal{X}|^{n}}^{n})$.
Map $x_{1}^{n}$ to $y_{1}:=\arg\max_{y\in\mathcal{Y}}Q_{Y}(y)$;
map $x_{2}^{n}$ to $y_{2}:=\arg\max_{y\in\mathcal{Y}}\left\{ Q_{Y}(y)-P_{X}^{n}(x_{1}^{n})1\left\{ y=y_{1}\right\} \right\} $;
...; map $x_{|\mathcal{X}|^{n}-|\mathcal{Y}|-2}^{n}$ to $y_{|\mathcal{X}|^{n}-|\mathcal{Y}|-1}:=\arg\max_{y\in\mathcal{Y}}\bigl\{ Q_{Y}(y)-\sum_{i=1}^{|\mathcal{X}|^{n}-|\mathcal{Y}|-1}P_{X}^{n}(x_{i}^{n})1\left\{ y=y_{i}\right\} \bigr\}$.
Map the remaining $|\mathcal{Y}|+1$ sequences $x_{j}^{n},|\mathcal{X}|^{n}-|\mathcal{Y}|\leq j\leq|\mathcal{X}|^{n}$
to sequences in $\mathcal{Y}$ in a similar way as in the proof of
Proposition \ref{prop:NUdiscontinuous}, such that $\sup_{y\in\mathbb{R}}\left|\sum_{y'\le y}\sum_{i=1}^{j}P_{X}^{n}(x_{i}^{n})1\left\{ y'=y_{i}\right\} -G_{Y}(y)\right|$
is minimized for $|\mathcal{X}|^{n}-|\mathcal{Y}|\leq j\leq|\mathcal{X}|^{n}$.

For this mapping, observe that for any $1\leq j\leq|\mathcal{X}|^{n}$,
\begin{align*}
\max_{y\in\mathcal{Y}}\left\{ Q_{Y}(y)-\sum_{i=1}^{j}P_{X}^{n}(x_{i}^{n})1\left\{ y=y_{i}\right\} \right\}  & \geq\frac{\sum_{y\in\mathcal{Y}}\left\{ Q_{Y}(y)-\sum_{i=1}^{j}P_{X}^{n}(x_{i}^{n})1\left\{ y=y_{i}\right\} \right\} }{|\mathcal{Y}|}\\
 & =\frac{1-\sum_{i=1}^{j}P_{X}^{n}(x_{i}^{n})}{|\mathcal{Y}|}\\
 & =\frac{\sum_{i=j+1}^{|\mathcal{X}|^{n}}P_{X}^{n}(x_{i}^{n})}{|\mathcal{Y}|}.
\end{align*}
This implies that if $P_{X}^{n}(x_{j+1}^{n})\leq\frac{\sum_{i=j+1}^{|\mathcal{X}|^{n}}P_{X}^{n}(x_{i}^{n})}{|\mathcal{Y}|}$,
then $\sum_{i=1}^{j+1}P_{X}^{n}(x_{i}^{n})1\left\{ y_{j+1}=y_{i}\right\} \leq Q_{Y}(y_{j+1})$.
Therefore, the following claim holds. 
\begin{claim}
\label{claim:If--holds-1}If $P_{X}^{n}(x_{j+1}^{n})\leq\frac{\sum_{i=j+1}^{|\mathcal{X}|^{n}}P_{X}^{n}(x_{i}^{n})}{|\mathcal{Y}|}$
holds for $1\leq j\leq m$ (for some integer $m$), then $\sum_{i=1}^{m+1}P_{X}^{n}(x_{i}^{n})1\left\{ y=y_{i}\right\} \leq Q_{Y}(y)$
for all $y\in\mathcal{Y}$. 
\end{claim}
Next we prove the following claim. 
\begin{claim}
\label{claim:If--holds}$P_{X}^{n}(x_{j+1}^{n})\leq\frac{\sum_{i=j+1}^{|\mathcal{X}|^{n}}P_{X}^{n}(x_{i}^{n})}{|\mathcal{Y}|}$
holds for $1\leq j\leq|\mathcal{X}|^{n}-|\mathcal{Y}|-2$ and for
$n\geq|\mathcal{Y}|\max_{1\le i\leq|\mathcal{X}|-1}\frac{P_{X}(x_{i})}{P_{X}(x_{i+1})}$. 
\end{claim}
We split the proof into two cases. 
\begin{itemize}
\item \textbf{Case 1 ($1\leq j\leq|\mathcal{X}|^{n}-n-2$):} For $1\leq j\leq|\mathcal{X}|^{n}-n-2$,
denote $k$ as the smallest integer such that $k\ge j+2$ and $T_{x_{k}^{n}}\neq T_{x_{j+1}^{n}}$.
Then we have 
\begin{equation}
P_{X}^{n}(x_{j+1}^{n})\leq P_{X}^{n}(x_{k}^{n})\max_{1\le i\leq|\mathcal{X}|-1}\frac{P_{X}(x_{i})}{P_{X}(x_{i+1})}.\label{eq:-22}
\end{equation}
This is because if $x_{k}^{n}$ and $x_{j+1}^{n}$ are different in
only one component, say the $i$th components $x_{k,i}$ and $x_{j+1,i}$,
then by the generation process of $x_{1}^{n},x_{2}^{n},...,x_{|\mathcal{X}|^{n}}^{n}$,
we have $\frac{P_{X}(x_{k,i})}{P_{X}(x_{j+1,i})}\geq\min_{1\le i\leq|\mathcal{X}|-1}\frac{P_{X}(x_{i+1})}{P_{X}(x_{i})}$.
Hence 
\[
P_{X}^{n}(x_{j+1}^{n})=P_{X}^{n}(x_{k}^{n})\frac{P_{X}(x_{j+1,i})}{P_{X}(x_{k,i})}\leq P_{X}^{n}(x_{k}^{n})\max_{1\le i\leq|\mathcal{X}|-1}\frac{P_{X}(x_{i})}{P_{X}(x_{i+1})}.
\]
If $x_{k}^{n}$ and $x_{j+1}^{n}$ are different in more than one
components, then by the generation process of $x_{1}^{n},x_{2}^{n},...,x_{|\mathcal{X}|^{n}}^{n}$,
we have $P_{X}^{n}(x_{j+1}^{n})\geq P_{X}^{n}(x_{k}^{n})\geq P_{X}^{n}(\tilde{x}_{k}^{n})$
for any sequence $\tilde{x}_{k}^{n}$ such that $\tilde{x}_{k}^{n}$
are different from $x_{j+1}^{n}$ in only one component. Hence by
the same argument above, we obtain 
\[
P_{X}^{n}(x_{j+1}^{n})\leq P_{X}^{n}(\tilde{x}_{k}^{n})\max_{1\le i\leq|\mathcal{X}|-1}\frac{P_{X}(x_{i})}{P_{X}(x_{i+1})}\leq P_{X}^{n}(x_{k}^{n})\max_{1\le i\leq|\mathcal{X}|-1}\frac{P_{X}(x_{i})}{P_{X}(x_{i+1})}.
\]
Therefore, \eqref{eq:-22} holds. Next, we prove \eqref{eq:-22} implies
Claim \ref{claim:If--holds}. 
\begin{align*}
\frac{P_{X}^{n}(x_{j+1}^{n})}{\sum_{i=j+1}^{|\mathcal{X}|^{n}}P_{X}^{n}(x_{i}^{n})} & \leq\frac{P_{X}^{n}(x_{j+1}^{n})}{\sum_{i=k}^{|\mathcal{X}|^{n}}P_{X}^{n}(x_{i}^{n})}\\
 & \leq\frac{P_{X}^{n}(x_{k}^{n})\max_{1\le i\leq|\mathcal{X}|-1}\frac{P_{X}(x_{i})}{P_{X}(x_{i+1})}}{\sum_{i=k}^{|\mathcal{X}|^{n}}P_{X}^{n}(x_{i}^{n})}
\end{align*}
According to the definition of $x_{k}^{n}$, we have $\mathcal{T}_{T_{x_{k}^{n}}}\subseteq\left\{ x_{k}^{n},x_{k+1}^{n},...,x_{|\mathcal{X}|^{n}}^{n}\right\} $.
Furthermore, each type class has at least $n$ elements, hence $\left|\mathcal{T}_{T_{x_{k}^{n}}}\right|\geq n$.
Therefore, 
\begin{align}
\frac{P_{X}^{n}(x_{k}^{n})\max_{1\le i\leq|\mathcal{X}|-1}\frac{P_{X}(x_{i})}{P_{X}(x_{i+1})}}{\sum_{i=k}^{|\mathcal{X}|^{n}}P_{X}^{n}(x_{i}^{n})} & \leq\frac{P_{X}^{n}(x_{k}^{n})\max_{1\le i\leq|\mathcal{X}|-1}\frac{P_{X}(x_{i})}{P_{X}(x_{i+1})}}{nP_{X}^{n}(x_{k}^{n})}\nonumber \\
 & =\frac{\max_{1\le i\leq|\mathcal{X}|-1}\frac{P_{X}(x_{i})}{P_{X}(x_{i+1})}}{n}\label{eq:-13}
\end{align}
For $n\geq|\mathcal{Y}|\max_{1\le i\leq|\mathcal{X}|-1}\frac{P_{X}(x_{i})}{P_{X}(x_{i+1})}$,
we have \eqref{eq:-13} $\leq\frac{1}{|\mathcal{Y}|}$. Therefore,
$P_{X}^{n}(x_{j+1}^{n})\leq\frac{\sum_{i=j+1}^{|\mathcal{X}|^{n}}P_{X}^{n}(x_{i}^{n})}{|\mathcal{Y}|}$. 
\item \textbf{Case 2 ($|\mathcal{X}|^{n}-n-1\leq j\leq|\mathcal{X}|^{n}-|\mathcal{Y}|-2$):}
For $|\mathcal{X}|^{n}-n-1\leq j\leq|\mathcal{X}|^{n}-|\mathcal{Y}|-2$,
we know $x_{j+1}^{n},...,x_{|\mathcal{X}|^{n}-1}^{n}$ belong to a
same type class, $P(x_{j+1}^{n})=...=P(x_{|\mathcal{X}|^{n}-1}^{n})=\left(P_{X}(x_{|\mathcal{X}|})\right)^{n-1}P_{X}(x_{|\mathcal{X}|-1})$,
and $P(x_{|\mathcal{X}|^{n}}^{n})=\left(P_{X}(x_{|\mathcal{X}|})\right)^{n}$.
Hence we have $\frac{P_{X}^{n}(x_{j+1}^{n})}{\sum_{i=j+1}^{|\mathcal{X}|^{n}}P_{X}^{n}(x_{i}^{n})}\leq\frac{P_{X}^{n}(x_{j+1}^{n})}{|\mathcal{Y}|P_{X}^{n}(x_{j+1}^{n})}=\frac{1}{|\mathcal{Y}|}$. 
\end{itemize}
Combining the two cases above, we have Claim \ref{claim:If--holds}.
By Claim \ref{claim:If--holds-1}, we further have $\sum_{i=1}^{m+1}P_{X}^{n}(x_{i}^{n})1\left\{ y=y_{i}\right\} \leq Q_{Y}(y)$
for all $y\in\mathcal{Y}$ with $m=|\mathcal{X}|^{n}-|\mathcal{Y}|-2$.

Since the remaining $|\mathcal{Y}|+1$ sequences are mapped to sequences
in $\mathcal{Y}$ in a similar way as in the proof of Proposition
\ref{prop:NUdiscontinuous}, similar to Proposition \ref{prop:NUdiscontinuous},
here we can show that the output measure $P_{Y}$ induced by the mapping
satisfies $\left|P_{Y}-Q_{Y}\right|_{\mathrm{KS}}\leq\frac{1}{2}P_{X}(x_{|\mathcal{X}|-1})\left(P_{X}(x_{|\mathcal{X}|})\right)^{n-1}$.

\section{\label{sec:Proof-of-Theorem-discrete-1}Proof of Proposition \ref{prop:Udiscontinuous}}

\textbf{Universal Mapping: }For $X_{n}$, partition the real line
into intervals with the same length $\Delta_{n}$, i.e., $\bigcup_{i=-\infty}^{\infty}(i\Delta_{n},(i+1)\Delta_{n}]$.
We first simulate a uniform distribution on $[a,b]$ by mapping each
interval $(i\Delta_{n},(i+1)\Delta_{n}]$ into $[a,b]$ using the
linear function $x\mapsto a+\frac{b-a}{\Delta_{n}}(x-i\Delta_{n})$.
That is, the function used here is 
\[
f_{1}(x):=\sum_{i=-\infty}^{\infty}\left(a+\frac{b-a}{\Delta_{n}}(x-i\Delta_{n})\right)1\left\{ x\in(i\Delta_{n},(i+1)\Delta_{n}]\right\} .
\]
We then transform the output distribution to the target distribution
$Q_{Y}$, by using function $x\mapsto G_{Y}^{-1}\left(\frac{x-a}{b-a}\right)$,
where $G_{Y}^{-1}\left(t\right):=\min\left\{ y:G_{Y}(y)\geq t\right\} $.
Therefore, each $x\in(i\Delta_{n},(i+1)\Delta_{n}]$ is mapped to
$G_{Y}^{-1}\left(\frac{1}{\Delta}(x-i\Delta)\right)$. Hence the final
mapping is 
\[
f(x):=\sum_{i=-\infty}^{\infty}G_{Y}^{-1}\left(\frac{1}{\Delta_{n}}(x-i\Delta_{n})\right)1\left\{ x\in(i\Delta_{n},(i+1)\Delta_{n}]\right\} .
\]
For such a universal simulator, we have 
\begin{align}
\left|P_{Y_{n}}-Q_{Y}\right|_{\mathrm{KS}} & \leq\left|P_{f_{1}(X_{n})}-\mathrm{Unif}([a,b])\right|_{\mathrm{KS}}\label{eq:-27}\\
 & =\sup_{z\in[a,b]}\left|\sum_{i=-\infty}^{\infty}\left[F_{X_{n}}\left(i\Delta_{n}+\Delta_{n}\frac{z-a}{b-a}\right)-F_{X_{n}}(i\Delta_{n})\right]-\frac{z-a}{b-a}\right|\nonumber \\
 & =\sup_{x\in[0,\Delta_{n}]}\left|\sum_{i=-\infty}^{\infty}\left[F_{X_{n}}(i\Delta_{n}+x)-F_{X_{n}}(i\Delta_{n})\right]-\frac{x}{\Delta_{n}}\right|\nonumber \\
 & =\sup_{x\in[0,\Delta_{n}]}\left|\sum_{i=-\infty}^{\infty}\left[F_{X_{n}}((i+1)\Delta_{n})-F_{X_{n}}(i\Delta_{n})\right]\left(\frac{F_{X_{n}}(i\Delta_{n}+x)-F_{X_{n}}(i\Delta_{n})}{F_{X_{n}}((i+1)\Delta_{n})-F_{X_{n}}(i\Delta_{n})}-\frac{x}{\Delta_{n}}\right)\right|\nonumber \\
 & \leq\sum_{i=-\infty}^{\infty}\left[F_{X_{n}}((i+1)\Delta_{n})-F_{X_{n}}(i\Delta_{n})\right]\sup_{x\in[0,\Delta_{n}]}\left|\frac{F_{X_{n}}(i\Delta_{n}+x)-F_{X_{n}}(i\Delta_{n})}{F_{X_{n}}((i+1)\Delta_{n})-F_{X_{n}}(i\Delta_{n})}-\frac{x}{\Delta_{n}}\right|\nonumber \\
 & \leq\sup_{x_{1}:F_{X_{n}}(x_{1}+\Delta_{n})>F_{X_{n}}(x_{1})}\sup_{x\in(x_{1},x_{1}+\Delta_{n}]}\left|\frac{F_{X_{n}}(x_{1}+x)-F_{X_{n}}(x_{1})}{F_{X_{n}}(x_{1}+\Delta_{n})-F_{X_{n}}(x_{1})}-\frac{x}{\Delta_{n}}\right|\nonumber \\
 & \to0,\textrm{ as }n\to\infty,\nonumber 
\end{align}
where \eqref{eq:-27} follows from Lemma \ref{lem:continuousisharder}.

\section{\label{sec:Proof-of-Theorem-Markov}Proof of Theorem \ref{thm:UMarkov}}

We still use a type-based mapping scheme (similar to that used in
the proof of Theorem \ref{thm:Udiscrete}) to prove Theorem \ref{thm:UMarkov}.

Here we adopt the notation from \cite{martin2010twice}. Assume the
Markov process $\boldsymbol{X}$ starts from the fixed initial state
$(x_{-k+1},x_{-k+2},...,x_{0})$. The $k$-th order Markov type $T_{x^{n}}$
of a sequence $x^{n}\in\mathcal{X}^{n}$ is defined as the number
of occurrences in $x^{n}$ of each string $\boldsymbol{s}\in\mathcal{X}^{k+1}$,
denoted $n_{x^{n}}(\boldsymbol{s})$, namely

\[
n_{x^{n}}(\boldsymbol{s})=\left|\left\{ i:1\leq i\le n,(x_{i-k},...,x_{i-1},x_{i})=\boldsymbol{s}\right\} \right|
\]
where $\left|\cdot\right|$ denotes cardinality. We use $T_{k}$ to
denote a $k$-th order Markov type of sequences in $\mathcal{X}^{n}$.
For a type $T_{k}$, the $k$-th order Markov type class $\mathcal{T}_{T_{k}}$
is the set of all sequences $x^{n}\in\mathcal{X}^{n}$ that have the
same type $T_{k}$. Obviously, all sequences in a type class have
a equal probability, i.e, $P(x^{n})=P(\tilde{x}^{n})$ for all $x^{n},\tilde{x}^{n}\in\mathcal{T}_{T_{k}}$.
That is, under the condition $X^{n}\in\mathcal{T}_{T_{k}}$, it is
uniformly distributed over the type class $\mathcal{T}_{T_{k}}$,
regardless of the distribution of the Markov process. Furthermore,
the set of $k$-th order Markov types of sequences in $\mathcal{X}^{n}$
is denoted as $\mathcal{P}_{k,n}\left(\mathcal{X}\right):=\left\{ T_{x^{n}}:x^{n}\in\mathcal{X}^{n}\right\} $.
It has been shown that $\left|\mathcal{P}_{k,n}\left(\mathcal{X}\right)\right|\leq\left(n+1\right)^{\left|\mathcal{X}\right|^{k+1}}$
in \cite{martin2010twice}.

Now we construct a mapping $f$ that maps the uniform random vector
on $\mathcal{T}_{T_{k}}$ to a random variable such that $\sup_{y\in\mathbb{R}}\left|F_{Y}(y|T_{k})-G_{Y}(y)\right|$
is minimized. Here $F_{Y}(y|T_{k})$ denotes the CDF of the output
random variable for the uniform random vector on $\mathcal{T}_{T_{k}}$.
Since the probability values of uniform random vectors are all equal
to $\left|\mathcal{T}_{T_{k}}\right|^{-1}$, $\sup_{y\in\mathbb{R}}\left|F_{Y}(y|T_{k})-G_{Y}(y)\right|=\frac{1}{2}\left|\mathcal{T}_{T_{k}}\right|^{-1}$.
Following the same steps as \eqref{eq:-20}-\eqref{eq:-21}, we obtain
\begin{align*}
\left|F_{Y}(y)-G_{Y}(y)\right| & \dotle\max_{x^{n}}P(x^{n}).
\end{align*}

The converse part follows from Theorem \ref{prop:NUdiscrete}, since
even non-universal simulation cannot make the approximation error
decay faster than $\max_{x^{n}}P(x^{n})$, hence universal simulation
cannot as well.

\section{\label{sec:Squeezing-Periodic-Functions}Properties on Squeezing
Periodic Functions}

For a function $g:[a,b]\to\mathbb{R}$ and a number $\Delta>0$, we
define a periodic function $g_{\Delta}$ induced by $g$ as 
\[
g_{\Delta}(x):=\sum_{i=-\infty}^{\infty}g\left(a+\frac{b-a}{\Delta}(x-i\Delta)\right)1\left\{ x\in(i\Delta,(i+1)\Delta]\right\} .
\]

Now we squeeze this periodic function in the $x$-axis direction by
letting $\Delta\to0$. It is easy to see that the limit $\lim_{\Delta\to0}g\left(a+\frac{b-a}{\Delta}(x-i\Delta)\right)$
of function $g$ does not exist. However, the integral $\lim_{\Delta\to0}\int f(x)g_{\Delta}(x)dx$
 for any integrable function $f$ exists. Moreover, this limit is
equal to the product of the integral of $f(x)$ and the normalized
integral of $g_{\Delta}(x)$ (or $g\left(x\right)$). That is, $f(x)$
and $g_{\Delta}(x)$ are asymptotically uncorrelated as $\Delta\to0$. 

Define 
\begin{align*}
L_{\Delta} & :=\int_{-\infty}^{\infty}f(x)g_{\Delta}(x)dx
\end{align*}
and 
\begin{align*}
L & :=\frac{1}{b-a}\int_{-\infty}^{\infty}f(x)dx\int_{a}^{b}g\left(x'\right)dx'.
\end{align*}
\begin{lem}
\label{lem:SqueezeTheorem}Assume $f(x)$ and $g(x)$ are arbitrary
integrable functions, and $\left|g\left(x\right)\right|$ is bounded
a.e. Then $f(x)$ and $g_{\Delta}(x)$ are asymptotically uncorrelated
as $\Delta\to0$. That is, 
\begin{align*}
\lim_{\Delta\to0}L_{\Delta} & =L.
\end{align*}
\end{lem}
\begin{rem}
More specifically, it holds that 
\begin{align*}
 & \left|L_{\Delta}-L\right|\le\esssup_{x}\left|g\left(x\right)\right|\sum_{i=-\infty}^{\infty}\int_{i\Delta}^{(i+1)\Delta}\left|f(x)-\frac{1}{\Delta}\int_{i\Delta}^{(i+1)\Delta}f(x)dx\right|dx.
\end{align*}
\end{rem}
\begin{rem}
If $f(x)$ is bounded and continuous a.e. on an interval $[c,d]$
(i.e., Riemann-integrable), and $f(x)=0,x\in[c,d]$, then the condition
$\sup_{x}\left|g\left(x\right)\right|$ is finite can be relaxed to
that $\int_{a}^{b}\left|g\left(x\right)\right|dx$ is finite. Furthermore,
for this case, Lemma \ref{lem:SqueezeTheorem} also holds for $g\left(x\right)=g_{1}(x)+\sum_{i=-\infty}^{\infty}c_{i}\delta(x-x_{i})$
with $x_{i}\in[a,b]$ such that $\int_{a}^{b}\left|g_{1}(x)\right|dx+\sum_{i=-\infty}^{\infty}c_{i}$
is finite, where $\delta(\cdot)$ denotes the Dirac delta function. 
\end{rem}
The proof of Lemma \ref{lem:SqueezeTheorem} is provided in Appendix
\ref{sec:Proof-of-Theorem-Squeeze}.

Now we generalize Lemma \ref{lem:SqueezeTheorem} by considering $g:[a,b]\times\mathbb{R}\to\mathbb{R}$
to be a real multivariate function. For such $g$ and a number $\Delta>0$,
we define a periodic function $g_{\Delta}$ induced by $g$ as 
\[
g_{\Delta}(x,y):=\sum_{i=-\infty}^{\infty}g\left(a+\frac{b-a}{\Delta}(x-i\Delta),y\right)1\left\{ x\in(i\Delta,(i+1)\Delta]\right\} .
\]
Define 
\begin{align*}
L_{\Delta}(y) & :=\int_{-\infty}^{\infty}f(x)g_{\Delta}(x,y)dx
\end{align*}
and 
\begin{align*}
L(y) & :=\frac{1}{b-a}\int_{-\infty}^{\infty}f(x)dx\int_{a}^{b}g\left(x',y\right)dx'.
\end{align*}
\begin{lem}
\label{lem:SqueezeTheorem-1-1}Assume $f(x)$ and $g(x,y)$ are arbitrary
integrable functions, and $\int_{-\infty}^{\infty}\left|g\left(x,y\right)\right|dy$
is bounded a.e. Then $f(x)$ and $g_{\Delta}(x,y)$ are asymptotically
uncorrelated under the $L_{1}$-norm distance as $\Delta\to0$. That
is, 
\begin{equation}
\lim_{\Delta\to0}\int_{-\infty}^{\infty}\left|L_{\Delta}(y)-L(y)\right|dy=0.\label{eq:-10}
\end{equation}
Furthermore, \eqref{eq:-10} also holds for $g\left(x,y\right)=g_{1}(x,y)+g_{3}(x)\delta(y-g_{2}(x))$
such that $g_{2}(x)$ is a differentiable a.e. function and $\frac{\mathrm{d}}{\mathrm{d}x}g_{2}(x)\neq0$
for almost every $x\in[a,b]$ and $\sup_{x}\int_{-\infty}^{\infty}\left|g\left(x,y\right)\right|dy=\sup_{x}\left\{ \int_{-\infty}^{\infty}\left|g_{1}\left(x,y\right)\right|dy+\left|g_{3}(x)\right|\right\} $
is finite. 
\end{lem}
\begin{rem}
\label{rem:More-specifically,-it}More specifically, it holds that
\begin{align}
 & \int_{-\infty}^{\infty}\left|L_{\Delta}(y)-L(y)\right|dy\nonumber \\
 & \le\esssup_{x}\int_{-\infty}^{\infty}\left|g\left(x,y\right)\right|dy\sum_{i=-\infty}^{\infty}\int_{i\Delta}^{(i+1)\Delta}\left|f(x)-\frac{1}{\Delta}\int_{i\Delta}^{(i+1)\Delta}f(x)dx\right|dx.\label{eq:-24}
\end{align}
If $g\left(x,y\right)=g_{1}(x,y)+g_{3}(x)\delta(y-g_{2}(x))$, then
\eqref{eq:-24} still holds with $\esssup_{x}\int_{-\infty}^{\infty}\left|g\left(x,y\right)\right|dy$
replaced by $\esssup_{x}\left|g_{3}(x)\right|$. 
\end{rem}
\begin{rem}
If $f(x)$ is bounded and continuous a.e. on an interval $[c,d]$,
and $f(x)=0,x\in[c,d]$, then the condition that $\int_{-\infty}^{\infty}\left|g\left(x,y\right)\right|dy$
is bounded a.e. can be relaxed to that $\int_{-\infty}^{\infty}\int_{a}^{b}\left|g\left(x,y\right)\right|dxdy$
is finite. Furthermore, for this case, Lemma \ref{lem:SqueezeTheorem-1-1}
also holds for $g\left(x,y\right)=g_{1}(x,y)+g_{3}(x)\delta(y-g_{2}(x))$
such that $g_{2}(x)$ is a differentiable a.e. function and $g_{2}'(x)\neq0$
for almost every $x\in[a,b]$ and $\int_{-\infty}^{\infty}\int_{a}^{b}\left|g\left(x,y\right)\right|dxdy=\int_{-\infty}^{\infty}\int_{a}^{b}\left|g_{1}\left(x,y\right)\right|dxdy+\int_{a}^{b}\left|g_{3}(x)\right|dx$
is finite. 
\end{rem}
The proof of Lemma \ref{lem:SqueezeTheorem-1-1} is similar to that
of Lemma \ref{lem:SqueezeTheorem}, and hence omitted. Furthermore,
Lemmas \ref{lem:SqueezeTheorem} and \ref{lem:SqueezeTheorem-1-1}
can be extended to multivariant function cases.

\subsection{\label{sec:Proof-of-Theorem-Squeeze}Proof of Lemma \ref{lem:SqueezeTheorem} }

Define 
\begin{align*}
L_{\Delta} & =\int_{-\infty}^{\infty}f(x)g_{\Delta}(x)dx\\
 & =\sum_{i=-\infty}^{\infty}\int_{i\Delta}^{(i+1)\Delta}f(x)g\left(a+\frac{b-a}{\Delta}(x-i\Delta)\right)dx
\end{align*}
and 
\begin{align*}
 & \sum_{i=-\infty}^{\infty}\int_{i\Delta}^{(i+1)\Delta}\left(\frac{1}{\Delta}\int_{i\Delta}^{(i+1)\Delta}f(x)dx\right)g\left(a+\frac{b-a}{\Delta}(x-i\Delta)\right)dx\\
 & =\sum_{i=-\infty}^{\infty}\left(\frac{1}{\Delta}\int_{i\Delta}^{(i+1)\Delta}f(x)dx\right)\int_{i\Delta}^{(i+1)\Delta}g\left(a+\frac{b-a}{\Delta}(x-i\Delta)\right)dx\\
 & =\sum_{i=-\infty}^{\infty}\left(\frac{1}{\Delta}\int_{i\Delta}^{(i+1)\Delta}f(x)dx\right)\frac{\Delta}{b-a}\int_{a}^{b}g\left(x'\right)dx'\\
 & =\frac{1}{b-a}\int_{-\infty}^{\infty}f(x)dx\int_{a}^{b}g\left(x'\right)dx'\\
 & =L.
\end{align*}
Then we bound $\left|L_{\Delta}-L\right|$ as follows. 
\begin{align}
 & \left|L_{\Delta}-L\right|\nonumber \\
 & =\biggl|\sum_{i=-\infty}^{\infty}\int_{i\Delta}^{(i+1)\Delta}f(x)g\left(a+\frac{b-a}{\Delta}(x-i\Delta)\right)dx\nonumber \\
 & \qquad-\sum_{i=-\infty}^{\infty}\int_{i\Delta}^{(i+1)\Delta}\left(\frac{1}{\Delta}\int_{i\Delta}^{(i+1)\Delta}f(x)dx\right)g\left(a+\frac{b-a}{\Delta}(x-i\Delta)\right)dx\biggr|\nonumber \\
 & =\left|\sum_{i=-\infty}^{\infty}\int_{i\Delta}^{(i+1)\Delta}\left(f(x)-\frac{1}{\Delta}\int_{i\Delta}^{(i+1)\Delta}f(x)dx\right)g\left(a+\frac{b-a}{\Delta}(x-i\Delta)\right)dx\right|\nonumber \\
 & \leq\sum_{i=-\infty}^{\infty}\int_{i\Delta}^{(i+1)\Delta}\left|f(x)-\frac{1}{\Delta}\int_{i\Delta}^{(i+1)\Delta}f(x)dx\right|\left|g\left(a+\frac{b-a}{\Delta}(x-i\Delta)\right)\right|dx\nonumber \\
 & \leq\esssup_{x}\left|g\left(x\right)\right|\sum_{i=-\infty}^{\infty}\int_{i\Delta}^{(i+1)\Delta}\left|f(x)-\frac{1}{\Delta}\int_{i\Delta}^{(i+1)\Delta}f(x)dx\right|dx\nonumber \\
 & =\esssup_{x}\left|g\left(x\right)\right|\sum_{i=-\infty}^{\infty}\int_{i\Delta}^{(i+1)\Delta}2\left[f(x)-\frac{1}{\Delta}\int_{i\Delta}^{(i+1)\Delta}f(x)dx\right]^{+}dx\nonumber \\
 & =\esssup_{x}\left|g\left(x\right)\right|\int_{-\infty}^{\infty}2\left[f(x)-f_{\Delta}(x)\right]^{+}dx,\label{eq:-1}
\end{align}
where $\left[z\right]^{+}=\max\{z,0\}$, and
\[
f_{\Delta}(x):=\sum_{i=-\infty}^{\infty}\frac{1}{\Delta}\int_{i\Delta}^{(i+1)\Delta}f(x)dx1\left\{ x\in(i\Delta,(i+1)\Delta]\right\} .
\]

Observe that $\left[f(x)-f_{\Delta}(x)\right]^{+}\leq f(x)$ and $f(x)$
is integrable, and moreover,
\begin{align}
 & \lim_{\Delta\to0}\sum_{i=-\infty}^{\infty}\frac{1}{\Delta}\int_{i\Delta}^{(i+1)\Delta}f(x)dx1\left\{ x_{0}\in(i\Delta,(i+1)\Delta]\right\} \nonumber \\
 & =\lim_{\Delta\to0}\frac{1}{\Delta}\int_{\left\lfloor \frac{x_{0}}{\Delta}\right\rfloor \Delta}^{\left(\left\lfloor \frac{x_{0}}{\Delta}\right\rfloor +1\right)\Delta}f(x)dx\nonumber \\
 & =f(x_{0}),\;\mathrm{a.e.},\label{eq:-9}
\end{align}
where \eqref{eq:-9} follows by Lebesgue's differentiation theorem
\cite[Thm. 7.7]{rudin2006real}. Hence by Lebesgue's dominated convergence
theorem \cite[Thm. 1.34]{rudin2006real}, \eqref{eq:-1} converges
to zero as $\Delta\to0$.

Therefore, $\lim_{\Delta\to0}L_{\Delta}$ exists and equals $L$.

\subsection*{Acknowledgments}

The author is deeply indebted to Prof. Vincent Y. F. Tan for making
helpful comments which greatly improve this paper. The author also
acknowledge Prof. Hideki Yagi for discussing on the extension to Rényi
divergence measures. The author is supported in part by a Singapore
National Research Foundation (NRF) National Cybersecurity R\&D Grant
(R-263-000-C74-281 and NRF2015NCR-NCR003-006), and in part by a National
Natural Science Foundation of China (NSFC) under Grant (61631017).
 \bibliographystyle{unsrt}
\bibliography{ref}

\end{document}